\documentclass{article}
\usepackage[a4paper,margin=1in]{geometry}

\usepackage{amsfonts}
\usepackage{amsthm}
\usepackage{amsmath}
\usepackage{amssymb}
\usepackage{xcolor}
\usepackage{comment}
\usepackage{enumerate}
\usepackage{mathrsfs}

\newtheorem{Thm}{Theorem}[section]
\newtheorem{Lem}[Thm]{Lemma}
\newtheorem{rem}[Thm]{Remark}
\newtheorem{Def}[Thm]{Definition}

\newtheorem{Prop}[Thm]{Proposition}
\newtheorem{Cor}[Thm]{Corollary}

\newtheorem*{Thm*}{Theorem}

\def\sideremark#1{\ifvmode\leavevmode\fi\vadjust{\vbox to0pt{\vss
 \hbox to 0pt{\hskip\hsize\hskip1em
 \vbox{\hsize3cm\tiny\raggedright\pretolerance10000
 \noindent #1\hfill}\hss}\vbox to8pt{\vfil}\vss}}}

\numberwithin{equation}{section}

\newenvironment{p}{{\it Proof.\ }}{\hfill$\square$\par}

\begin{document}

\title{Stability of Minkowski inequality for nearly spherical sets}

\author{Yi Wang 
\thanks{3400 N Charles St, Baltimore, MD, 21218, Department of Mathematics, Johns Hopkins University, ywang261@jhu.edu}
, Shuhan Yang \thanks{3400 N Charles St, Baltimore, MD, 21218, Department of Mathematics, Johns Hopkins University, syang180@jhu.edu}}

% \affil[2]{Department of Mathematics, Johns Hopkins University\\
% \texttt{syang180@jhu.edu}}

\date{\today}

\maketitle

\begin{abstract}
In this paper, we study the stability of Minkowski inequality for nearly spherical domains that are $C^1$ close to the ball. We show the stability inequalities between the  positive part of the $\sigma_k$ curvature integrals for $C^1$ perturbations of a ball; we also establish the stability inequalities for axially symmetric $C^1$ perturbations of a ball. 
%Moreover, for Minkowski inequality between $\sigma_k(h)$ integral and the volume of the enclosed domain, we obtain sharp stability inequality as well. 
Finally, we construct a counterexample, illustrating that the inequalities become invalid if we do not compensate the integral with the negative part of the curvature. 
Our work generalizes Glaudo’s results \cite{GLAUDO2022108595} on the mean curvature integral to the fully nonlinear cases. 
\end{abstract}

\section{Introduction} 
In this paper, we study the Minkowski inequalities for quermassintegrals on domains $\mathbb{K}$ that are $C^{1}\ \epsilon$-perturbations of the Euclidean ball. Building upon Glaudo’s seminal work on the Minkowski inequalities for nearly spherical domains, this study employs analytical techniques to deepen the understanding of $\sigma_k$ and its geometric implications. More precisely, we adopt Glaudo’s approach of decomposing functions into low- and high-frequency components, discuss the case of $\sigma_k$, and derive an analogous inequality for the $\sigma_k$ version. Additionally, counterexamples and further extensions are provided to highlight the nuanced behavior of $\sigma_k$ in diverse geometric settings.

Let $\Omega \subset \mathbb{R}^{n+1}$ be a bounded convex set. Denote the $m$-dimensional Hausdorff measure in $\mathbb{R}^{n+1}$ by $\mathcal{H}^{m}$. Consider the set
$$
\Omega + t B := \{x + t y \mid x \in \Omega,\ y \in B\},
$$
for $t > 0$. By a theorem of Minkowski, its volume is a polynomial of degree $n+1$ in $t$, given by
$$
\operatorname{Vol}(\Omega + tB) = \mathcal{H}^{n+1}(\Omega + tB) = \sum_{k=0}^{n+1} C_{n+1}^{k} W_k(\Omega) t^{k},
$$
where the coefficients $W_k(\Omega)$ for $k = 0, \ldots, n+1$ depend on the set $\Omega$, and
$$
C_{n+1}^{k} = \frac{(n+1)!}{k!(n+1-k)!}.
$$
The $k$-th quermassintegral $V_k$ is defined as a multiple of the coefficient $W_{n+1-k}(\Omega)$:
$$
V_k(\Omega) := \frac{\omega_k}{\omega_{n+1}} W_{n+1-k}(\Omega),
$$
where $\omega_k$ denotes the volume of the unit $k$-ball. For any domain $\Omega$, we have $V_{n+1}(\Omega) = \mathcal{H}^{n+1}(\Omega)$.

If $\Omega$ has a smooth boundary $M := \partial \Omega$, the quermassintegrals can also be expressed in terms of curvature integrals. Let $h_{ij}$ denote the second fundamental form on $M$, and let $\sigma_m(h)$, for $m = 0, \ldots, n$, represent the $m$-th elementary symmetric function of the eigenvalues of $h$ (with the convention $\sigma_0(h) = 1$). Then, for $k = 1, \ldots, n+1$,
$$
V_{n+1-k}(\Omega) := \frac{(n+1-k)!(k-1)!}{(n+1)!} \cdot \frac{\omega_{n+1-k}}{\omega_{n+1}} \int_{M} \sigma_{k-1}(h)\, d\mu_{M}.
$$

From this definition, one observes that
$$
V_0(\Omega) = 1, \qquad
V_n(\Omega) = \frac{1}{(n+1)\omega_{n+1}}\, \mathcal{H}^{n}( \partial \Omega ),
$$
where $\mathcal{H}^n(\partial \Omega)$ denotes the area of the boundary.

As a consequence of the classical Aleksandrov--Fenchel inequalities (see \cite{Ale37,Ale38}), one obtains the following family of inequalities: if $\Omega$ is a convex domain in $\mathbb{R}^{n+1}$ with smooth boundary, then for $0 \le l \le n$,
\begin{equation}\label{33333}
\left( \frac{V_{l+1}(\Omega)}{V_{l+1}(B)} \right)^{\frac{1}{l+1}}
\le
\left( \frac{V_l(\Omega)}{V_l(B)} \right)^{\frac{1}{l}},
\end{equation}
which is equivalent to
\begin{equation}\label{44444}
\left( \int_{M} \sigma_{k-1}(h)\, d\mu_{M} \right)^{\frac{1}{n-k+1}}
\le
C \left( \int_{M} \sigma_{k}(h)\, d\mu_{M} \right)^{\frac{1}{n-k}},
\end{equation}
for $k = n - l$, $1 \le k \le n$, where $C = C(k,n)$ denotes the sharp constant attained when $M$ is a round sphere. In particular, when $k = 0$, inequality (\ref{33333}) reduces to the classical isoperimetric inequality:
$$
\left( \mathcal{H}^{n+1}(\Omega) \right)^{\frac{n}{n+1}}
\le
\frac{\omega_{n+1}}{n+1}\, \mathcal{H}^n( \partial \Omega ).
$$
By induction, (\ref{33333}) also yields
\begin{equation}\label{999999}
\frac{
\left( \int_{\partial K} \sigma_k(h) \, d\mathcal{H}^{n} \right)^{\frac{1}{n+1-k}}
}{
|K|^{\frac{1}{n+1}}
}
\ge
\frac{
\left( \int_{\partial B_1} \sigma_k(\mathbb{I}_{n}) \, d\mathcal{H}^{n} \right)^{\frac{1}{n+1-k}}
}{
|B_1|^{\frac{1}{n+1}}
},
\quad \text{for all } k = 1, \ldots, n.
\end{equation}

The inequalities (\ref{33333}) for convex domains were originally proved using Minkowski’s theory of mixed volumes. The argument in \cite{Ale37,Ale38} relies heavily on the convexity assumption. Since then, many alternative methods have been developed to establish these inequalities for convex domains. (See Hörmander’s book \cite{hormander1994notions}.)
%some avoiding the notion of mixed volumes altogether (see Hörmander’s book \cite{hormander1994notions} for further discussion). 
%On the other hand, attemps are made to generalize these inequalities from convex domains to nonconvex domains. In particular, a family of domains called $k$-convex domains are naturally introduced to study this problem. 
On the other hand, many attempts have been made to generalize these inequalities from convex domains to nonconvex ones. In particular, a natural class of domains, called 
$k$-convex domains, has been introduced to study this problem. Although our main focus is not on $k$-convex domains, we recall their definition here for completeness.

\begin{Def}
For $\Omega \subset \mathbb{R}^{n+1}$, we say the boundary $M := \partial \Omega$ is $k$-convex if the second fundamental form $h_{ij}(x)$ satisfies $\sigma_m(h) > 0$ for all $x \in M$ and for all $1 \le m \le k$, that is,
$$
\Gamma_k^{+} := \{ A \in M_{n \times n} \mid \sigma_m(A) > 0,\ \forall\, 1 \le m \le k \}.
$$
\end{Def}

With this notation, $n$-convexity corresponds to standard convexity, and $1$-convexity is referred to as mean convexity.

Guan-Li \cite{GUAN20091725}  applied a fully nonlinear flow to study inequality (\ref{44444}) for $m$-convex domains. Specifically, one evolves the hypersurface $M := \partial \Omega \subset \mathbb{R}^{n+1}$ along the flow
\begin{equation}\label{66666}
\frac{\partial X_t}{\partial t}
=
\frac{\sigma_{k-1}(h)}{\sigma_k(h)}\, \nu,
\end{equation}
where $\nu$ is the unit outer normal of the hypersurface $M$. The key observation is that the ratio
\begin{equation}\label{77777}
\frac{
\left( \int_M \sigma_{k-1}(h)\, d\mu_M \right)^{\frac{1}{n-k+1}}
}{
\left( \int_M \sigma_k(h)\, d\mu_M \right)^{\frac{1}{n-k}}
}
\end{equation}
is monotonically increasing along the flow (\ref{66666}). Therefore, if the solution of (\ref{66666}) exists for all time $t > 0$ and converges to a round sphere (up to rescaling), then the sharp inequality (\ref{44444}) follows. This strategy applies to certain classes of domains—for instance, convex ones.

In the special case $k = 1$, (\ref{66666}) corresponds to the inverse mean curvature flow, which has been extensively studied in the literature, for example by Evans and Spruck \cite{Evans1991635}, and by Huisken and Ilmanen \cite{Huisken2001353}. In this case, under the additional assumption that $\Omega$ is outward minimizing, \cite{Huisken2001353} implies that the sharp inequality (\ref{44444}) holds for $k=1$. Another class of domains for which this strategy works is the set of star-shaped and strictly $k$-convex domains. In this setting, Gerhardt \cite{Gerhardt1990299} and Urbas \cite{Urbas1990355} independently proved that the flow (\ref{66666}) exists for all $t$ and converges to a round sphere, enabling Guan-Li \cite{GUAN20091725} to establish \eqref{44444} for star-shaped and $k$-convex domains. 

%\begin{Thm}(\cite{GUAN20091725})Suppose $\Omega$ is a smooth star-shaped domain in $\mathbb{R}^{n+1}$ with $k$-convex boundary. Then inequality (\ref{44444}) holds for all $1 \le m \le k$, with equality if and only if $\Omega$ is a ball.\end{Thm}

In general, without additional assumptions on the domain, one expects singularities to develop along the flow (\ref{66666}), and thus the flow typically does not exist for all time. Instead, other approaches have been developed. Below, we list several results, if not all, that generalize these inequalities to various families of nonconvex domains without relying on flow methods.

\begin{itemize}

    %The case $k = 1$ of (\ref{33333}) for general domains was first proved, with a non-sharp constant, in \cite{Michael1973SobolevAM}. Subsequently, \cite{CASTILLON201079} improved the constant (via optimal transport methods), and more recently, Brendle obtained the sharp constant for hypersurfaces with boundary using the ABP method (see \cite{brendle2020isoperimetricinequalityminimalsubmanifold}); the ABP method itself is classical, tracing back to Cabré–Xavier and ultimately to Gromov’s idea behind the Knothe map.

    % --- (B) IMCF + outward-minimizing：2001 → 2013 → 2018 → 2022 ---
    %\item For outward-minimizing domains (a stronger condition than mean convexity), the case $k = 1$ of (\ref{44444}) is due to Huisken (reported in \cite{alexandre_2013}, Theorem 5, Lemma 8; see also \cite{YWei18}, Theorem 1.1). The proof relies fundamentally on the (weak) inverse mean curvature flow developed in \cite{Huisken2001353}. The same result was also shown in \cite{AFM21} by an alternative proof that avoids the regularity difficulties arising when applying the inverse mean curvature flow in dimensions higher than 7. As a direct consequence, one deduces (1.3) with $k = 1$ for arbitrary domains $K \subset \mathbb{R}^n$.

    % --- (C) Flow methods (Gerhardt–Urbas)：1990 → 1990 → 2009 ---
   % \item For a $k$-convex star-shaped domain, inequality (\ref{44444}) was proved in \cite{GUAN20091725} via a suitable flow of the boundary hypersurface that converges to a sphere (see also \cite{Gerhardt1990299, Urbas1990355}).
    
    \item For $(k+1)$-convex domains in $\mathbb{R}^{n+1}$, \eqref{44444} was established in \cite{CHANG2013335} using the method of optimal transport, albeit with a non-sharp constant (i.e., the right-hand side replaced by a positive constant). Later, inequality \eqref{999999} was proven by Chang--Wang \cite{ChangWangIMRN}, and Qiu \cite{Qiu} using optimal transport techniques. Recently, Agostiniani--Fogagnolo--Mazzieri \cite{AFM21} gave another proof of the Minkowski inequality $k=1$ case. %for outward minimizing domains. 
    The proof relies on the discovery of effective monotonicity formulas holding along the level set flow of the $p$-capacitary potentials associated with $\Omega$.

   % \item In dimension $n = 4$, inequality (\ref{44444}) holds for axially symmetric domains \cite{Dal}.

    \item Michael--Simon \cite{Michael1973SobolevAM} proved \eqref{33333} for general domains and $k = 1$ with a non-sharp constant.  Subsequently, Castillon \cite{CASTILLON201079} gave an alternative proof via optimal transport methods. In \cite{wang2013michaelsimoninequalitieskthmean}, the first author of this paper  generalizes Michael--Simon type inequalities between the fully nonlinear curvature quantities $\sigma_{k-1}(L)$ and $\sigma_{k}(L)$ for $(k+1)$-convex domains. Recently, Brendle \cite{brendle2020isoperimetricinequalityminimalsubmanifold}, Brendle--Eichmair \cite{brendle2023proofmichaelsimonsobolevinequalityusing} proved the sharp Michael--Simon inequality, using the ABP  method and optimal transport, respectively, for submanifolds of codimension at most $2$. These works also gave an alternative proof of Michael--Simon inequality with non-sharp constant in arbitrary codimension.

We note that the Alexandrov–Bakelman–Pucci (ABP) method and other auxiliary map approaches to proving the isoperimetric inequality trace back to an elegant argument of Gromov (see, for example, p. 47 of \cite{Chavel}) using the Knothe map, as well as to the work of Cabré \cite{Cabre1, Cabre2} employing solutions of linear equations. See also \cite{TRUDINGER1994411} for extensions to general quermassintegrals.

    %In \cite{wang2013michaelsimoninequalitieskthmean}, the author generalizes Michael--Simon type inequalities to relate the fully nonlinear curvature quantities $\sigma_{k-1}(L)$ and $\sigma_{k}(L)$. The proof relies on optimal    transport maps to connect curvature quantities defined on the boundary of the domain.

    %In \cite{brendle2023proofmichaelsimonsobolevinequalityusing}, the authors gave an alternative proof of the Michael-Simon-Sobolev inequality using techniques from optimal transport. The inequality is sharp for submanifolds of codimension 2.

    % --- (F) Nonlinear potential theory：2022 ---
    
    % --- (G) Nearly spherical stability：2024 ---

    %In paper \cite{WangCaroline}, the authors proved stability estimates for the Alexandrov--Fenchel inequality in the setting of nearly spherical domains. Their results yield quantitative isoperimetric inequalities and extend the stability theory for quermassintegrals beyond the convex or star-shaped category, improving upon earlier work of Guan--Li.

    %In \cite{Scheuer+2025+381+405}, the author employs a specific inverse-curvature quantity previously studied in \cite{Gerhardt1990299, Urbas1990355} and, together with the almost-umbilicity theorem of \cite{De_Rosa_2021}, establishes a stability result for the Alexandrov–Fenchel inequalities under the $k$-convex assumption. The argument, however, requires a form of $C^2$-closeness for the hypersurface.
   \item Regarding the stability inequalities, VanBlargan--Wang \cite{WangCaroline} proved stability estimates for the Alexandrov--Fenchel inequality in the setting of nearly spherical domains in the spirit of Fuglede \cite{Fuglede}. Scheuer \cite{Scheuer+2025+381+405}
employed the inverse curvature flow together with the almost-umbilicity theorem of \cite{De_Rosa_2021} to derive a stability result for the Alexandrov–Fenchel inequalities. Glaudo \cite{GLAUDO2022108595} proved the stability inequality for integral of the positive part of the mean curvature $H$ for $C^1$ approximations of the sphere.

\end{itemize}

 %The validity of (\ref{33333}) for $k$-convex domains remains open when the sharp constant is required. 
 In this paper, we prove some stability results for domains that are $C^1$ close (and $C^3$ bounded) to the ball. We don't need to assume the boundary is $k$-convex, and prove the stability inequality for the quantity 
$\int_{\partial \Omega}  \sigma_k(h)+ \sum_{j=1}^{k} \sigma_j(h)^-d\mu$. Moreover, we show in Theorem \ref{thm:9.1}  
that the stability inequality for $\int_{\partial \Omega}  \sigma_k(h)d\mu$ does not hold. 

\begin{Thm} \label{thm:5.1}
Suppose $\Omega=\left\{\left(1+u\left(\frac{x}{|x|}\right)\right) x: x \in B\right\} \subseteq \mathbb{R}^{n+1}$ where $u \in C^{3}(\partial B)$, $\operatorname{Vol}(\Omega)=\operatorname{Vol}(B)$, and $\operatorname{bar}(\Omega)=0$. Assume $n \geq 5$, $\|u\|_{C^2}\leq M$ for some fixed $M>0$. Then, for $k\leq [\frac{n-1}{2}]$, there exists an $\epsilon(k,n,M)>0$ such that if $\|u\|_{C^1}<\epsilon(k,n,M)$, we have:
\begin{equation}
\int_{\partial \Omega}  \sigma_k(h)+ \sum_{j=1}^{k} \sigma_j(h)^-d\mu-\int_{\partial B} C_n^k dA \geq d(k,n)\int_{\partial B} |\nabla u|^2 dA
\end{equation}

\noindent for some constant $d(k,n)>0$. 
\end{Thm}

\begin{Thm} \label{thm:5.2}
Suppose $\Omega=\left\{\left(1+u\left(\frac{x}{|x|}\right)\right) x: x \in B\right\} \subseteq \mathbb{R}^{n+1}$, $I_{j'}(\Omega)=I_{j'}(B)$, where $u \in C^{3}(\partial B)$, $j'$ is an even number smaller than $k$, and $\operatorname{bar}(\Omega)=0$. Assume $n \geq 5$, $\|u\|_{C^2}\leq M$ for some fixed $M>0$. Then for even $k> n-j'-2$, there exists an $\epsilon(j',k,n,M)>0$ such that , if $\|u\|_{C^1}<\epsilon(j',k,n,M)$, $\sum_{m=1}^{j'} (-1)^m \sigma_m(h)\geq 0$ and $\sum_{m=j'+1}^{k} (\frac{C_n^{k+1}}{C_n^m}+(-1)^{k-m})\sigma_m(h)\geq 0$, we have:
\begin{equation}
\int_{\partial \Omega}  \sigma_k(h)d\mu-\int_{\partial B} C_n^k dA \geq d(j',k,n)\int_{\partial B} |\nabla u|^2 dA
\end{equation}

\noindent for some constant $d(j',k,n)>0$. 
\end{Thm}

\begin{rem}The conditions $\sum_{m=1}^{j'} (-1)^m \sigma_m(h)\geq 0$ is satisfied
by the standard sphere. This is because, on the standard sphere, the second fundamental form is proportional to the identity, i.e., all principal curvatures are equal to 1. In this case, the elementary symmetric functions $\sigma_m(h)$ are just the binomial coefficients $C_n^m$. Hence, the first alternating sum becomes alternating sums of binomial coefficients, which are known to be nonnegative due to combinatorial identities. Specifically,
\begin{equation}
\sum_{m=1}^{j'} (-1)^m C_{n}^{m} \geq 0,
\end{equation}
for choices of $j'$ stated in the theorem, which follows from standard results on alternating sums of binomial coefficients. 
%The proof is given in Section 2.
\end{rem}
% Requires: \usepackage{amsmath,amssymb}
\begin{rem}
We consider the following inequality involving the elementary symmetric functions $\sigma_m(h)$:
\begin{equation}
    \sum_{m=j'+1}^{k} \left( \frac{C_n^{k+1}}{C_n^m} + (-1)^{k-m} \right) \sigma_m(h) \geq 0
\end{equation}

\noindent On the unit sphere $\mathbb{S}^n$, we have $\sigma_m(h) = C_n^m$. Substituting this into the inequality allows it to be simplified to:
\[
    (k - j') C_n^{k+1} + C_{n-1}^k - C_{n-1}^{j'} \geq 0 \quad (\text{for even } k, j')
\]
For instance, the case $j'=0$, $k=n-1$, where $n$ is odd, satisfies this assumption.

%+ Section 4 result for $k$

%\noindent However, this inequality \textbf{does not hold} for the following parameters:\[ n=7, \quad k=6, \quad j'=2\] These parameters satisfy the conditions $j' < k < n$ and $j'+k+2 > n$. Therefore, this inequality does not necessarily hold for the standard sphere.
\end{rem}
\noindent If $D^2 u$ satisfies a condition, we also derive stability inequality, which is presented in the following.
\begin{Thm}\label{thm:4.5}
    Suppose $\Omega=\left\{\left(1+u\left(\frac{x}{|x|}\right)\right) x: x \in B\right\} \subseteq \mathbb{R}^{n+1}$ where $u \in C^{3}(\partial B)$, $\operatorname{Vol}(\Omega)=\operatorname{Vol}(B)$, and $\operatorname{bar}(\Omega)=0$. Assume $n \geq 5$, $\|u\|_{C^2}\leq  M$ for some constant $M$, $(-1)^{m} \sigma_m(D^2 u) \geq (-1)^m C_n^m$ for integers $1\le m \le k$. Then, for $k<n-1$, there exists an $\epsilon(k,n,M)>0$ such that if $\|u\|_{C^1}<\epsilon(k,n,M)$, we have
\begin{equation}
\int_{\partial \Omega} \sigma_k(h) d\mu-\int_{\partial B} C_n^k dA \geq d(k,n)\int_{\partial B} |\nabla u|^2 dA
\end{equation}
\noindent for some constant $d(k,n)>0$.
\end{Thm}
\begin{Thm} \label{thm:3}
Suppose 
$\Omega = \left\{\bigl(1+u\!\left(\frac{x}{|x|}\right)\bigr)x : x \in B \right\} \subseteq \mathbb{R}^{n+1}$,  
where $u \in C^{3}(\partial B)$, $\operatorname{bar}(\Omega)=0$, and  
$I_{0}(\Omega) = I_{0}(B)$. Assume $n \geq 5$ and $\|u\|_{C^2} \leq M$ for some fixed $M>0$. Then for any $\delta>0$ and $k < [\frac{n}{2}]$, there exists an $\epsilon=\epsilon(k,n,M)>0$ such that if $\|u\|_{C^1}<\epsilon$, we have
\begin{equation}
\int_{\partial \Omega}
  \sigma_k (h)
  + \sum_{j=1}^{k} \sigma_j(h)^-
   d\mu
- \int_{\partial B} C_{n}^{k} \, dA 
\;\;\geq\; -\delta.
\end{equation}
\end{Thm}

\begin{Thm} \label{thm:5.3} Suppose $\Omega=\left\{\left(1+u\left(\frac{x}{|x|}\right)\right) x: x \in B\right\} \subseteq \mathbb{R}^{n+1}$ where $u \in C^{3}(\partial B)$, $I_{0}(\Omega) = I_{0}(B)$, and $\operatorname{bar}(\Omega)=0$. Assume $n \geq 5$, $\|u\|_{C^2}\leq M$ for some fixed $M>0$. Then there exists an $\epsilon(k,n,M)>0$ such that if $\|u\|_{C^1}<\epsilon(k,n,M)$ and $\Omega$ is axially symmetric (i.e., it is invariant under rotations around a fixed line), we have:
\begin{equation}
\int_{\partial \Omega}  \sigma_k(h)+ \sum_{j=1}^{k} \sigma_j(h)^-d\mu-\int_{\partial B} C_n^k dA \geq d(k,n)\int_{\partial B} |\nabla u|^2 dA
\end{equation}

\noindent for some constant $d(k,n)>0$. 

\end{Thm}

\noindent Finally, we construct an example in the following theorem, showing that the results would fail if the term
$\int_{\partial \Omega}  \sigma_k(h)+ \sum_{j=1}^{k} \sigma_j(h)^-d\mu$
is replaced by term $\int_{\partial \Omega} \sigma_k(h) d\mu$ in Theorem \ref{thm:5.1}, Theorem \ref{thm:3} and Theorem \ref{thm:5.3}.
\begin{Thm}\label{thm:9.1}
 Given $n\geq 3$, for any $\epsilon >0$, for $k$ odd, there is domain $\Omega \subset \mathbb R^{n+1}$ which is a $C^1$ $\epsilon$-pertubation of a ball and such that

 \begin{equation}
\int_{\partial \Omega} \sigma_k(h) d\mu <-1.
 \end{equation}
 \end{Thm}

\paragraph{Organization of the Paper.}
The paper is organized as follows. Section 2 provides the fundamental preliminaries, introducing the definitions of the elementary symmetric functions $\sigma_k(\cdot)$, the notion of $k$-convexity, and the Newton transformation tensors $T_k$, together with several basic analytic identities that will be used repeatedly. 
Section 3 is devoted to some preliminary lemmas, including a lemma about low--high frequency decomposition from earlier work of Glaudo and some combinatoric identities and inequalities. % In particular, these lemmas establish the essential framework for the \emph{low--high frequency decomposition} that plays a central role in our later analysis and stability estimates. 
In Section 4, we impose a reasonable boundedness assumption on $\sigma_k(D^2u)$, and prove, in this setting, Theorem~\ref{thm:4.5}. 
This assumption, though technical, is natural from the geometric point of view and allows the main ideas of the argument to appear transparently. 
In particular, this section highlights the significance of the low--high frequency decomposition, showing how it captures the dominant curvature contributions and clarifies the mechanism behind the subsequent proofs. 
Sections~5--8 are devoted to proving the four main theorems of the paper, each addressing a distinct curvature inequality or rigidity result derived from the general framework established earlier. 
Finally, in Section 9 we show, by constructing a counterexample, we show that Theorem \ref{thm:5.1}, Theorem \ref{thm:3} and Theorem \ref{thm:5.3} are false for perturbations of a ball %if the negative part term $\sum_{j=1}^k\sigma_j(h)^{-}$ is removed from the inequalities and only $\sigma_k(h)$ is retained.
if we have replaced the term
$\int_{\partial \Omega}  \sigma_k(h)+ \sum_{j=1}^{k} \sigma_j(h)^-d\mu$
 by $\int_{\partial \Omega} \sigma_k(h) d\mu$.

\section{Preliminaries}
	
For $\lambda=\left(\lambda_{1}, \ldots, \lambda_{n}\right) \in \mathbb{R}^{n}$, we denote $\sigma_{k}(\lambda)$ as the $k$-th elementary symmetric polynomial of $\left(\lambda_{1}, \ldots, \lambda_{n}\right)$. That is, for $1 \leq k \leq n$,
	
\begin{equation}
\sigma_{k}(\lambda)=\sum_{i_{1}<i_{2}<\ldots<i_{n}} \lambda_{i_{1}} \lambda_{i_{2}} \cdots \lambda_{i_{k}},
\end{equation}
	
\noindent and
\begin{equation}
\sigma_{0}(\lambda)=1.
\end{equation}
 
\noindent This leads to a natural generalization of the mean curvature of a surface.
	
\begin{Def}
Suppose $\Omega$ is a smooth, bounded domain in $\mathbb{R}^{n+1}$. For $x \in M:=\partial \Omega$, the $k$-th mean curvature of $M$ at $x$ is $\sigma_{k}(\lambda)$, where $\lambda=\left(\lambda_{1}(x), \ldots, \lambda_{n}(x)\right)$ are the principal curvatures of $M$ at $x$.
\end{Def}
	
\noindent Observe that in this definition, $\sigma_{1}(\lambda)$ is the mean curvature and $\sigma_{n}(\lambda)$ is the Gaussian curvature. When $\left(\lambda_{1}, \ldots, \lambda_{n}\right)$ are the eigenvalues of a matrix $A=\left\{A_{j}^{i}\right\}$, we denote $\sigma_{k}(A)=\sigma_{k}(\lambda)$, which can be equivalently calculated as
	
\begin{equation}
\sigma_{k}(A)=\frac{1}{k!} \delta_{i_{1} \cdots i_{k}}^{j_{1} \cdots j_{k}} A_{j_{1}}^{i_{1}} \cdots A_{j_{k}}^{i_{k}}
\end{equation}
\noindent using the Einstein convention to sum over repeated indices. So, if $h$ denotes the second fundamental form of $M$, we can use this expression for $\sigma_{k}(h)$ to compute the $k$-th mean curvature of $M$. %Throughout this paper, we will be working with a family of surfaces where, for $0<j \leq k$, $\sigma_{j}(h) \geq 0$ at each point. Such surfaces are called $k$-convex.
A useful operator related to $\sigma_{k}$ is the Newton transformation tensor $\left[T_{k}\right]_{i}^{j}$.
	
\begin{Def}
The Newton transformation tensor, $\left[T_{k}\right]_{i}^{j}$, of $n \times n$ matrices $\left\{A_{1}, \ldots, A_{k}\right\}$ is defined as
		
\begin{equation}
\left[T_{k}\right]_{i}^{j}\left(A_{1}, \ldots, A_{k}\right):=\frac{1}{k!} \delta_{i i_{1} \cdots i_{k}}^{j j_{1} \cdots j_{k}}\left(A_{1}\right)_{j_{1}}^{i_{1}} \cdots\left(A_{k}\right)_{j_{k}}^{i_{k}}.
\end{equation}
		
\noindent When $A_{1}=A_{2}=\ldots=A_{k}=A$, we denote $\left[T_{k}\right]_{i}^{j}(A):=\left[T_{k}\right]_{i}^{j}(A, \ldots, A)$.
\end{Def}
	
\noindent A related operator is $\Sigma_{k}$, which is the polarization of $\sigma_{k}$.
	
\begin{Def} Let $A_{1}, \ldots, A_{k}$ be a collection of $n \times n$ symmetric matrices. We denote
		
\begin{equation}
\begin{aligned}
\Sigma_{k}\left(A_{1}, \ldots, A_{k}\right)  :&=\left(A_{1}\right)_{j}^{i}\left[T_{k-1}\right]_{i}^{j}\left(A_{2}, \ldots, A_{k}\right) \\
& =\frac{1}{(k-1)!} \delta_{i_{1} \cdots i_{k}}^{j_{1} \cdots j_{k}}\left(A_{1}\right)_{j_{1}}^{i_{1}} \cdots\left(A_{k}\right)_{j_{k}}^{i_{k}}.
\end{aligned}
\end{equation}
\end{Def}
\noindent Two useful identities are
\begin{equation}
\sigma_{k}(A)=\frac{1}{k} \Sigma_{k}(A, \ldots, A)=\frac{1}{k} A_{j}^{i}\left[T_{k-1}\right]_{i}^{j}(A)
\end{equation}
\noindent and	
\begin{equation}
\frac{\partial \sigma_{k}(A)}{\partial A_{j}^{i}}=\frac{1}{k}\left[T_{k-1}\right]_{i}^{j}(A).
\end{equation}
	
\noindent We will also use the identity
	
\begin{equation}
A_{s}^{j}\left[T_{m}\right]_{j}^{i}(A)= \sigma_{m+1}(A)\delta_{s}^{i}-\left[T_{m+1}\right]_{s}^{i}(A).
\end{equation}
	
\begin{Lem}
$\nabla_{j}\left[T_{m}\right]_{i}^{j}\left(D^{2} u\right)=-(n-m) u_{j}\left[T_{m-1}\right]_{i}^{j}\left(D^{2} u\right).$
\end{Lem}
	
\noindent \begin{p} We compute
		
\begin{equation}
		\begin{aligned}
			\nabla_{j}\left[T_{m}\right]_{i}^{j}\left(D^{2} u\right) & =\frac{1}{m!} \nabla_{j} \delta_{i i_{1} i_{2} \ldots i_{m}}^{j j_{1} j_{2} \ldots j_{m}} u_{j_{1}}^{i_{1}} u_{j_{2}}^{i_{2}} \cdots u_{j_{m}}^{i_{m}} \\
			& =\frac{m}{m!} \delta_{i i_{1} i_{2} \ldots i_{m}}^{j j_{1} j_{2} \ldots j_{m}}\left(\nabla_{j} u_{j_{1}}^{i_{1}}\right) u_{j_{2}}^{i_{2}} \cdots u_{j_{m}}^{i_{m}}.
		\end{aligned}
\end{equation}
		
\noindent Note that
		
\begin{equation}
		\delta_{i i_{1} i_{2} \ldots i_{m}}^{j j_{1} j_{2} \ldots j_{m}}\left(\nabla_{j} u_{j_{1}}^{i_{1}}\right) u_{j_{2}}^{i_{2}} \cdots u_{j_{m}}^{i_{m}}=-\delta_{i i_{1} i_{2} \ldots i_{m}}^{j_{1} j j_{2} \ldots j_{m}}\left(\nabla_{j} u_{j_{1}}^{i_{1}}\right) u_{j_{2}}^{i_{2}} \cdots u_{j_{m}}^{i_{m}}.
\end{equation}
		
\noindent We obtain
		
\begin{equation}
\begin{aligned}
			\nabla_{j}\left[T_{m}\right]_{i}^{j}\left(D^{2} u\right) & =\frac{1}{2(m-1)!} \delta_{i i_{1} i_{2} \ldots i_{m}}^{j j_{1} j_{2} \ldots j_{m}}\left(\nabla_{j} u_{j_{1}}^{i_{1}}-\nabla_{j_{1}} u_{j}^{i_{1}}\right) u_{j_{2}}^{i_{2}} \cdots u_{j_{m}}^{i_{m}} \\
			& =\frac{1}{2(m-1)!} \delta_{i l i_{1} \ldots i_{m-1}}^{j s j_{1} \ldots j_{m-1}}\left(u_{p} R_{s j}{}^{p l}\right) u_{j_{1}}^{i_{1}} \cdots u_{j_{m-1}}^{i_{m-1}},
\end{aligned}
\end{equation}
		
\noindent where $R_{s j}{} ^{p l}$ is the curvature tensor on $\Omega$. On $\mathbb{S}^{n}$, we know by the Gauss equation,
  
\begin{equation}
		R_{s j}{}^{p l}=h_{s}^{p} h_{j}^{l}-h_{j}^{p} h_{s}^{l}=\delta_{s}^{p} \delta_{j}^{l}-\delta_{j}^{p} \delta_{s}^{l}.
\end{equation}
\noindent Therefore,
		
	\begin{equation}
		\begin{aligned}
			\nabla_{j}\left[T_{m}\right]_{i}^{j}\left(D^{2} u\right) & =\frac{1}{(m-1)!} \frac{1}{2} u_{p}\left(\delta_{s}^{p} \delta_{j}^{l}-\delta_{j}^{p} \delta_{s}^{l}\right) \delta_{i l i_{1} \ldots i_{m-1}}^{j s j_{1} \ldots j_{m-1}} u_{j_{1}}^{i_{1}} \cdots u_{j_{m-1}}^{i_{m-1}} \\
			& =\frac{1}{(m-1)!} \frac{1}{2}\left(u_{s} \delta_{i j i_{1} \ldots i_{m-1}}^{j s j_{1} \ldots j_{m-1}} u_{j_{1}}^{i_{1}} \cdots u_{j_{m-1}}^{i_{m-1}}-u_{j} \delta_{i l i_{1} \ldots i_{m-1}}^{j l j_{1} \ldots j_{m-1}} u_{j_{1}}^{i_{1}} \cdots u_{j_{m-1}}^{i_{m-1}}\right) \\
			& =\frac{-1}{(m-1)!} u_{j} \delta_{i l i_{1} \ldots i_{m-1}}^{j l j_{1} \ldots j_{m-1}} u_{j_{1}}^{i_{1}} \cdots u_{j_{m-1}}^{i_{m-1}} \\
			& =-(n-m) u_{j}\left[T_{m-1}\right]_{i}^{j}\left(D^{2} u\right).
		\end{aligned}
	\end{equation}
\end{p}

\begin{Lem}
    For a symmetric matrix A, we have $A \cdot [T_{k-1}](A)=[T_{k-1}](A)\cdot A$. 
\end{Lem}
\noindent \begin{p} We argue by induction, assume that $A \cdot [T_{k-2}](A)=[T_{k-2}](A)\cdot A$ for $k\geq 2$, then
\begin{equation*}
\begin{aligned}
& A_j^j\left[T_{k-1}\right]_k^j(A)=A_j^i \left( \sigma_{k-1} (A)\delta_k^j-\left[T_{k-2}\right]_p^j(A) A_k^p\right) \\
& {\left[T_{k-1}\right]_j^i(A) A_k^j=\left(\sigma_{k-1}(A) \delta_j^i-\left[T_{k-2}\right]_p^i(A) A_j^p\right) \cdot A_k^j}.
\end{aligned}
\end{equation*}

\noindent Note that $A_j^i \sigma_{k-1} (A)\delta_k^j=A_k^i \sigma_{k-1}(A), \sigma_{k-1}(A) \delta_j^i A_k^j=\sigma_{k-1}(A) A_k^i$, so
\begin{equation*}
 A_j^i \sigma_{k-1}(A) \delta_k^j=\sigma_{k-1}(A) \delta_j^i A_k^j.
\end{equation*}

\noindent So we only need to compare $A_j^i\left[T_{k-2}\right]_p^j(A) A_k^p$ and $\left[T_{k-2}\right]_p^i(A) A_j^p A_k^j$.

\noindent By induction, 
\begin{equation*}
\begin{aligned}
&\quad A_j^i\left[T_{k-2}\right]_p^j(A) A_k^p=\left[T_{k-2}\right]_j^i(A) A_p^j A^p_k \\
& =\left[T_{k-2}\right]_p^i(A) \cdot\left(A^2\right)_k^p,
\end{aligned}
\end{equation*}

\noindent thus the result holds.
\end{p}

\begin{Lem}\label{lem:2.6}
		Suppose $\Omega \subseteq \mathbb{R}^{n+1}$ where $M=\{(1+u(x)) x: x \in \partial B\}$ and $u \in C^{2}(\partial B)$. Then 
\begin{align}
    \sigma_{k}\left(h\right) &= 
    \frac{1}{((1+u)^2+|\nabla u|^2)^{\frac{k+2}{2}}}\sum_{m=0}^{k} \frac{(-1)^{m} C_{n-m}^{k-m} \left((1+u)^{2} \sigma_{m}\left(D^{2} u\right) + \frac{n+k-2m}{n-m} u^{i} u_{j}\left[T_{m}\right]_{i}^{j}\left(D^{2} u\right)\right)}
    {(1+u)^{m}}.
\end{align}
\end{Lem}
\noindent \begin{p} See \cite{VW22}, Lemma 3.1.
\end{p}
%\noindent \textbf{Handy identity.}
\noindent We will present then a handy identity.
\begin{Lem}
For integers $n\ge 1$ and $0\le t\le n$,
$$
\sum_{r=0}^{t}(-1)^rC_{n}^{r}
= (-1)^tC_{n-1}^{t}.
$$
\end{Lem}
\noindent \begin{p}
Let $S_t:=\sum_{r=0}^{t}(-1)^rC_{n}^{r}$. Using Pascal’s identity
$$
C_{n}^{r}=C_{n-1}^{r}+C_{n-1}^{r-1}\quad(\text{with }C_{n-1}^{-1}=0),
$$
we get
$$
S_t=\sum_{r=0}^{t}(-1)^rC_{n-1}^{r}
  +\sum_{r=0}^{t}(-1)^rC_{n-1}^{r-1}.
$$
Set $s=r-1$ in the second sum (so $s=-1,\dots,t-1$); the $s=-1$ term vanishes, hence
$$
\sum_{r=0}^{t}(-1)^rC_{n-1}^{r-1}
=\sum_{s=0}^{t-1}(-1)^{s+1}C_{n-1}^{s}
=-\sum_{s=0}^{t-1}(-1)^{s}C_{n-1}^{s}.
$$
Therefore
$$
S_t
=\Bigg(\sum_{r=0}^{t}(-1)^rC_{n-1}^{r}\Bigg)
 -\Bigg(\sum_{s=0}^{t-1}(-1)^{s}C_{n-1}^{s}\Bigg)
=(-1)^tC_{n-1}^{t},
$$
as claimed.
\end{p}
We will derive two corollaries from this identity.
\begin{Cor}
\noindent 1) For $j'\ge 1$,
$$
\sum_{s=1}^{j'}(-1)^sC_{n}^{s}
=\sum_{s=0}^{j'}(-1)^sC_{n}^{s}-1
= (-1)^{j'}C_{n-1}^{j'}-1.
$$
In particular, if $j'$ is even and $j'<n$, then $\sum_{s=1}^{j'}(-1)^sC_{n}^{s}\ge 0$.

\medskip

\noindent 2) For $j'+1\le k\le n$,
$$
\sum_{m=j'+1}^{k}(-1)^{k-m}C_{n}^{m}
=\sum_{m=0}^{k}(-1)^{k-m}C_{n}^{m}
 -\sum_{m=0}^{j'}(-1)^{k-m}C_{n}^{m}
=C_{n-1}^{k}-(-1)^{k+j'}C_{n-1}^{j'}.
$$
If $j'$ is even, this becomes
$$
\sum_{m=j'+1}^{k}(-1)^{k-m}C_{n}^{m}
=\begin{cases}
C_{n-1}^{k}+C_{n-1}^{j'}>0, & k \text{ odd},\\[4pt]
C_{n-1}^{k}-C_{n-1}^{j'}, & k \text{ even}.
\end{cases}
$$
\end{Cor}
We remark that by unimodality of binomial coefficients, when $k+j'+2<n$, the sum $C_{n-1}^{k}-C_{n-1}^{j'}$
is nonnegative.
%\edz{Where is Corollary used? Is the last sentence about the remark more clear this way?}

\section{Important Lemmas}
	
\begin{Lem}\label{lem:3.1}
Let $0 \in U \subseteq \mathbb{R}$ be an open set and let $f, g: U \times U \rightarrow(0, \infty)$ be two positive $C^{1}$ functions such that $f(0,0)=g(0,0)=1$. Assume that $B(0, \bar{r}) \subseteq U$ for a certain $\bar{r}>0$. Let $u: \partial B \rightarrow(-1, \infty)$ be a function such that $\|u\|_{C^{1}} \leq \varepsilon$, with $\varepsilon<\bar{r}$. Decompose the function $u$ as $u=u_{1}+u_{2}$, where $u_{1}, u_{2}$ belong to the subspaces generated by the eigenfunctions of $-\Delta$ with eigenvalues respectively smaller and larger than a fixed $\lambda>0$. Then it holds
 \begin{equation}
    \begin{aligned}
       &\quad \int_{\partial B} f\left(u,|\nabla u|^{2}\right) \nabla^{2} u[\nabla u, \nabla u] 
        \\&\geq -\left(\frac{1}{2}+\omega(\varepsilon)\right) 
        \int_{\partial B} \operatorname{div}\left(g\left(u,|\nabla u|^{2}\right) \nabla u\right)\left|\nabla u_2\right|^{2}  \\
        &+ \omega(\varepsilon) 
        \int_{\partial B}\left(\left[\operatorname{div}\left(g\left(u,|\nabla u|^{2}\right) \nabla u\right)\right]^{+}+1\right)|\nabla u|^{2} .
    \end{aligned}
\end{equation}
\end{Lem}
\noindent
\begin{p}
		See \cite{GLAUDO2022108595}, Lemma 4.1.
	\end{p}
\begin{Lem}\label{lem:5.2}
For integers $n > k \ge j \ge 0$, the following identity holds:
\begin{equation} 
\label{eq:beta-identity}
\begin{split}
&\sum_{m=j}^{k} 
        \frac{1}{(m-j)!(k-m)!}
        \cdot \frac{1}{(m+1)(\,n-m\,)}
        \cdot (-1)^{\,j+m}
    \;\\
    =&\;\frac{1}{(k-j)!(n+1)}\{
    B(j+1,\, k-j+1) 
    \;+\; 
    (-1)^{\,k-j}\, B(n-k,\, k-j+1)\},
    \end{split}
\end{equation}
where $B(\cdot,\cdot)$ denotes the Beta function.

\noindent Moreover, the expression in \eqref{eq:beta-identity} is negative if and only if  
\[
j > n-k-1
\qquad\text{and}\qquad 
k-j \;\text{is odd}.
\]

\noindent In particular, when $j=0$ and $k<n-1$, the sign of
\begin{equation*}
\sum_{m=0}^k 
(-1)^m 
C_{n-m}^{\,k-m\,}
C_{n}^{m}
\frac{(n-k)(k+1)}{2(m+1)(n-m)}
\end{equation*}  
is the same as the sign of
\begin{equation*}
\sum_{m=0}^k 
\frac{1}{m!(k-m)!}
\cdot \frac{1}{(m+1)(n-m)} 
\cdot (-1)^m,
\end{equation*}
which is positive. This quantity will appear frequently in our later computations.
\end{Lem}
\noindent \begin{p}  For a given $j \geq 1$, note that

% 	\begin{equation}
% 	\sum_{m=j}^k \frac{1}{(m-j)!(k-m)!} \cdot \frac{1}{(m+1)(n-m)} \cdot (-1)^{j+m}.
% 	\end{equation}
	
	\begin{equation}
    \begin{aligned}
	&\quad \sum_{m=j}^k \frac{1}{(m-j)!(k-m)!} \cdot \frac{1}{(m+1)(n-m)} \cdot (-1)^{j+m}\\
	&= \frac{\sum_{s=0}^{k-j} C_{k-j}^s}{(k-j)!} \cdot (-1)^s \cdot \left(\frac{1}{s+j+1} - \frac{1}{s+j-n}\right) \cdot \frac{1}{n+1}.
    \end{aligned}
	\end{equation}
	
\noindent So we only need to consider
the sign of    
	\begin{equation}
	\sum_{s=0}^{k-j} C_{k-j}^s (-1)^s \cdot \left(\frac{1}{s+j+1} - \frac{1}{s+j-n}\right).
	\end{equation}
	
\noindent Note that $(1-r)^{k-j} r^j = \sum_{s=0}^{k-j} C_{k-j}^s \cdot (-1)^s \cdot r^{s+j}$.
	
\noindent Thus
	
\begin{equation}
	\sum_{s=0}^{k-j} C_{k-j}^s (-1)^s \cdot \frac{1}{s+j+1} = \int_0^1 (1-r)^{k-j} \cdot r^j \, dr,
\end{equation}
	
\noindent and similarly,
\begin{equation}
	\sum_{s=0}^{k-j} C_{k-j}^s (-1)^s \cdot \frac{1}{s+j-n} = -\int_1^{\infty} (1-r)^{k-j} \cdot r^{j-n-1} \, dr,
\end{equation}
if $k<n$.

\noindent Therefore
\begin{equation}\label{ghj}
	\begin{aligned}
		&\quad \sum_{s=0}^{k-j} C_{k-j}^s (-1)^s \cdot \left(\frac{1}{s+j+1} - \frac{1}{s+j-n}\right) \\
		& = \int_0^1 (1-r)^{k-j} r^j \, dr + \int_1^{+\infty} (1-r)^{k-j} \cdot r^{j-n-1} \, dr \\
		& = \int_0^1 (1-r)^{k-j} \cdot r^j \, dr + \int_0^1 (r-1)^{k-j} \cdot r^{n-k-1} \, dr.
	\end{aligned}
\end{equation}
	
\noindent So if $k-j$ is even, $\sum_{m=j}^k \frac{1}{(m-j)!(k-m)!} \cdot \frac{1}{(m+1)(n-m)} \cdot (-1)^{j+m} > 0$.
	
\noindent If $k-j$ is odd, then the result is
	
\begin{equation}
	\int_0^1 (1-r)^{k-j} \cdot \left(r^j - r^{n-k-1}\right) \, dr.
\end{equation}
	
\noindent If $j > n-k-1$, $\sum_{m=j}^k \frac{1}{(m-j)!(k-m)!} \cdot \frac{1}{(m+1)(n-m)} \cdot (-1)^{j+m} < 0$. If $j \leq n-k-1$, $\sum_{m=j}^k \frac{1}{(m-j)!(k-m)!} \cdot \frac{1}{(m+1)(n-m)} \cdot (-1)^{j+m} \geq 0$.\\
\noindent The case $j=0$ and $k<n-1$ follows directly from equation~(\ref{ghj}).

\end{p}
\section{Proofs under an additional assumption}
\noindent
\noindent
In this section we impose the additional boundedness assumptions
$$
(-1)^m \sigma_m(D^2 u)\;\geq\;(-1)^m C_{n}^{m}, \quad 1\le m\le k.
$$
These conditions are not assumed in the statement of Theorem~\ref{thm:5.1}, but they are introduced 
for heuristic purposes: they are the natural necessary conditions which ensure that
$$
\sigma_k(h)\;\ge 0 
$$
(up to an $O(\epsilon^2)$ term according to formula \eqref{eq:leading}).
This way the curvature functionals are nonnegative. Under these assumptions, 
the expansions become more transparent and reveal the crucial role of the low–high frequency 
decomposition of $u$ in establishing stability estimates. This motivates the refined argument 
given later without the extra assumptions.

\subsection{The $\sigma_2$ case}
\begin{Lem}\label{lem:4.1} 
We have
\begin{align}
\int_{\partial B} u^i u_j [T_2]_{i}^j (D^2 u) d \sigma =
\int_{\partial B} u^i u_j
\left[\sigma_2(D^2 u)\delta_{i}^j-[T_1]_{k}^j (D^2 u) u^k_{i}\right] d\sigma,
\end{align}
\noindent and
\begin{align}
\int_{\partial B} u^i u^j[T_1]_{i}^k (D^2 u)u_{kj}d\sigma \notag =&
\int_{\partial B} -  \frac{1}{2}|\nabla u|^2\sigma_2(D^2 u)
-\frac{1}{2}u^i u^j\left([T_1]_{i}^k (D^2 u)\right)_k u_{j}d\sigma \notag \\
=&\int_{\partial B} -  \frac{1}{2}|\nabla u|^2\sigma_2(D^2 u)+\frac{n}{2} |\nabla u|^4 
 d\sigma.
\end{align}
\end{Lem}

\begin{Lem}
    Suppose $\Omega=\left\{\left(1+u\left(\frac{x}{|x|}\right)\right) x: x \in B\right\} \subseteq \mathbb{R}^{n+1}$ and $u \in C^{3}(\partial B)$. If $\|u\|_{C^{1}}<\epsilon$, then

$$
\begin{aligned}
\int_{\partial \Omega} \sigma_{2}(h) d \mu & =\int_{\partial B}C_n^2+C_n^2(n-2) u+C_n^2\frac{(n-2)(n-2-1)}{2} u^{2} \\
& +\sum_{m=0}^{2}(-1)^{m}
C_{n-m}^{2-m}
\frac{(n-2)(2+1)}{2(m+1)(n-m)}|\nabla u|^{2} \sigma_{m}\left(D^{2} u\right) d A\\
&-\frac{3(n-4)(n-5)}{4}\int_{\partial B} u |\nabla u|^{2} \Delta u d A +O(\epsilon)\|\nabla u\|_{L^{2}}^{2}+O(\epsilon)\|u\|_{L^{2}}^{2} .
\end{aligned}
$$
\end{Lem}
\noindent
\begin{p}
    This proof closely follows the method used in Lemma 4.2 in \cite{VW22}. Most steps are straightforward except for equality (98) in their paper. For $m=1$,
\begin{align*}
\int_{\partial B} u^2 \sigma_1(D^2u) d \sigma=
-2\int_{\partial B} u|\nabla u|^2 dA=O(\epsilon)\|\nabla u\|_{L^{2}}^{2}.
\end{align*}

\noindent For $m=2$,
\begin{align*}
\int_{\partial B} u^{2} \sigma_{2}\left(D^{2} u\right) d A & =\frac{1}{2} \int_{\partial B}u^{2} u_{j}^{i}\left[T_{1}\right]_{i}^{j}\left(D^{2} u\right) d A \\
& =\frac{1}{2} \int_{\partial B} u^2 (\Delta u)^2 d A-\frac{1}{2} \int_{\partial B} u^2 u^i_j u^j_i d A.
\end{align*}

\noindent Note that
\begin{align*}
\int_{\partial B} u^2 (\Delta u)^2 d A=-\int_{\partial B} 2u |\nabla u|^2 \Delta u d A-\int_{\partial B} u^2 (\nabla u) \cdot (\nabla(\Delta u)) d A,
\end{align*}

\noindent and also
\begin{align*}
\int_{\partial B} u^2 u^i_j u^j_i d A=-\int_{\partial B} 2u u_j u^i u^j_i d A-\int_{\partial B} u^2 \nabla u\cdot \nabla(\Delta u)d A.
\end{align*}

\noindent Furthermore
\begin{align*}
\int_{\partial B} 2u u_j u^i u^j_i d A=\int_{\partial B} -|\nabla u|^4 d A-\int_{\partial B} u |\nabla u|^2 \Delta u d A.
\end{align*}

\noindent Combining these results together, we obtain
\begin{align*}
\int_{\partial B} u^{2} \sigma_{2}\left(D^{2} u\right) d A=-\frac{1}{2} \int_{\partial B} |\nabla u|^{4}  d A-\frac{3}{2} \int_{\partial B} u |\nabla u|^{2} \Delta u d A.
\end{align*}

\noindent Using these results from \cite{VW22}, adding the terms $\int_{\partial B} u^{2} \sigma_{2}\left(D^{2} u\right) d A$ and $\int_{\partial B} u^{2} \sigma_{1}\left(D^{2} u\right) d A$, we get

$$
\begin{aligned}
\int_{\partial \Omega} \sigma_{2}(h) d \mu & =\int_{\partial B}
C_n^2+C_n^2 (n-2) u+\left(\begin{array}{l}
n \\
2
\end{array}\right) \frac{(n-2)(n-2-1)}{2} u^{2} \\
& +\sum_{m=0}^{2}(-1)^{m}
C_{n-m}^{2-m}\frac{(n-2)(2+1)}{2(m+1)(n-m)}|\nabla u|^{2} \sigma_{m}\left(D^{2} u\right) d A\\
&+ \frac{(n-4)(n-5)}{2}\left(-\frac{1}{2} \int_{\partial B} |\nabla u|^{4}  d A-\frac{3}{2} \int_{\partial B} u |\nabla u|^{2} \Delta u d A\right)\\
&+O(\epsilon)\|\nabla u\|_{L^{2}}^{2}+O(\epsilon)\|u\|_{L^{2}}^{2} .
\end{aligned}
$$
\end{p}

\begin{Prop}
\label{s}
Suppose $\Omega=\left\{\left(1+u\left(\frac{x}{|x|}\right)\right) x: x \in B\right\} \subseteq \mathbb{R}^{n+1}$ where $u \in C^{3}(\partial B)$, $\operatorname{Vol}(\Omega)=\operatorname{Vol}(B)$, and $\operatorname{bar}(\Omega)=0$. Assume $n \geq 5$, $-M \leq \Delta u \leq n$ for some fixed $M \geq n > 0$. Additionally, assume that for sufficiently small $\epsilon(n,M)>0$, $\|u\|_{C^1}<\epsilon(n,M)$. Decompose the function $u$ as $u = u_1 + u_2$, where $u_1, u_2$ belong to the subspaces generated by the eigenfunctions of $-\Delta u$ with eigenvalues respectively smaller and larger than a fixed sufficiently large $\lambda(n) > 0$. Then

\begin{align*}
I_{2}(\Omega)-I_{2}(B) & \geq d(n) \int_{\partial B} |\nabla u|^2+O(\epsilon)\sigma_2(h)^- dA +O(\epsilon)\|\nabla u\|_{L^{2}}^{2}+O(\epsilon)\|u\|_{L^{2}}^{2},
\end{align*}\
\noindent for some constant $d(n)>0$.
\end{Prop}

\begin{proof}
Using the formula $\int_{\partial B} u =-\frac{n}{2} \int_{\partial B}  u^2+ O(\epsilon)\|u|_{L^2}^2$ in Proposition 4.3 in \cite{VW22} with the same proof structure, we get
 $$
\begin{aligned}
I_2(\Omega)-I_2(B) & =\int_{\partial B}
C_n^2
\frac{(n-2)(2+1)}{2 n}|\nabla u|^{2}-C_n^2\frac{(n-2)(2+1)}{2} u^{2} \\
& +\sum_{m=1}^{2}(-1)^{m}
C_{n-m}^{2-m}
\frac{(n-2)(2+1)}{2(m+1)(n-m)}|\nabla u|^{2} \sigma_{m}\left(D^{2} u\right) d A\\
&-\frac{3(n-4)(n-5)}{4}\int_{\partial B} u |\nabla u|^{2} \Delta u d A+O(\epsilon)\|\nabla u\|_{L^{2}}^{2}+O(\epsilon)\|u\|_{L^{2}}^{2}.
\end{aligned}
$$
\noindent Now, by Lemma \ref{lem:2.6}, we have
\begin{align*}
\sigma_2(h)\cdot ((1+u)^2+|\nabla u|^2)^2 &=\frac{n(n-1)}{2}(1+u)^2+\frac{(n-1)(n+2)}{2}|\nabla u|^2+\sigma_2(D^2u)\\
&+\frac{1}{(1+u)^2}u^i u_j [T_2]^j _i (D^2u)-(n-1)(1+u) \Delta u\\
&-\frac{n}{1+u} u^i u_j [T_1]^j _i(D^2u).
\end{align*}
\begin{align*}
\sigma_2(D^2u)-(n-1)\Delta u &\geq  (n-1)u \Delta u+\frac{n}{1+u} u^i u_j [T_1]^j _i(D^2u)\\
&-\frac{n(n-1)}{2}(1+u)^2-\frac{1}{(1+u)^2}u^i u_j [T_2]^j _i (D^2u)\\
&-\frac{(n-1)(n+2)}{2}|\nabla u|^2-\sigma_2(h)^-\cdot ((1+u)^2+|\nabla u|^2)^2.
\end{align*}

\noindent Substituting this inequality into the above expression and using $\|u\|_{C^1}<\epsilon(n,M)$, we obtain
\begin{align*}
I_{2}(\Omega)-I_{2}(B) & \geq \int_{\partial B}C_{n}^{2}\frac{(n-2)(2+1)}{2n}|\nabla u|^{2}
-C_{n}^{2}\frac{(n-2)(2+1)}{2} u^{2} \\
& +\frac{1}{2}|\nabla u|^2\Big((n-1)u \Delta u+\frac{n}{1+u} u^i u_j [T_1]^j _i(D^2u)-\frac{n(n-1)}{2}\\
&\quad -\frac{1}{(1+u)^2}u^i u_j [T_2]^j _i(D^2u)-\frac{3}{2}\sigma_2(h)^- +(2-\frac{1}{2}n) \Delta u\Big)dA\\
&-\frac{3(n-4)(n-5)}{4}\int_{\partial B} u |\nabla u|^{2} \Delta u \, dA
+O(\epsilon)\|\nabla u\|_{L^{2}}^{2}+O(\epsilon)\|u\|_{L^{2}}^{2}.
\end{align*}

\noindent By direct computation,
$$
\int_{\partial B} |\nabla u|^2 u^i u_j u_{i}^jdA
=\;-\frac{1}{4}\int_{\partial B}  |\nabla u|^4 \Delta u \, dA.
$$

\noindent Hence
\begin{align*}
\int_{\partial B} |\nabla u|^{2} u^i u_j[T_1]^j _i (D^2 u) d A
&= \int_{\partial B} \Big(|\nabla u|^{4}\Delta u - |\nabla u|^{2}u_j u_{ij}u_i\Big)\, dA \\
&= \frac{5}{4}\int_{\partial B} |\nabla u|^{4} \Delta u \, dA \\
&> \frac{5}{4}\int_{\partial B \cap \{\Delta u \leq 0\}} |\nabla u|^{4} \Delta u \, dA \\
&= O(\epsilon)\Big(\int_{\partial B} |\nabla u|^{2} \Delta u \, dA+\|\nabla u\|_{L^2}^2 \Big).
\end{align*}

\medskip
\noindent For the second Newton transform, we claim
$$
\int_{\partial B} |\nabla u|^{2}\, u^i u_j\,[T_2]^j_i(D^2 u)\, dA
=\frac{3}{2}\int_{\partial B} |\nabla u|^{4}\,\sigma_2(D^2 u)\, dA.
$$

\noindent Indeed, let $A:=D^2u$ and recall
$$
[T_1]_i^j(A)=\sigma_1(A) \delta_{i}^j-A_i^j, \qquad [T_2]_i^j(A)=\sigma_2(A) \delta_i^j-\sigma_1(A) A_i^j+A_i^kA_k^j=\sigma_2(A) \delta_i^j-A_i^k[T_1]_k^j(A),
$$
together with the standard identities
$$
\operatorname{div}([T_1](A))=0, \qquad [T_1](A)_i^jA^i_j=2\sigma_2(A).
$$

\noindent Using $[T_2]_i^j(A)=\sigma_2(A) \delta_i^j-A_i^k[T_1]_k^j(A)$, we decompose
$$
u^i u_j[T_2]^j_i(A)=\sigma_2(A)|\nabla u|^2-  A_i^k u_ku^l[T_1]_l^i(A).
$$
Thus
$$
\int_{\partial B} |\nabla u|^2 u^i u_j[T_2]^j_i
=\int_{\partial B} |\nabla u|^4 \sigma_2(A) \, dA
-\int_{\partial B} |\nabla u|^2  A_i^k u_k u^l [T_1]^i_l(A) \, dA.
$$

\noindent With $\phi:=|\nabla u|^2$, the weighted divergence identity
$$
0=\int_{\partial B}\operatorname{div}(\phi^2 T_1(A)\nabla u)\,dA
=\int_{\partial B}\big(4\phi A_i^k u_ku^l[T_1]_l^i(A) +2\phi^2\sigma_2(A)\big)\,dA
$$
yields
$$
\int_{\partial B}\phi A_i^k u_ku^l[T_1]_l^i(A) dA
=-\frac{1}{2}\int_{\partial B}\phi^2\sigma_2(A)\,dA.
$$

\noindent Substituting back gives
$$
\int_{\partial B} |\nabla u|^2 u^i u_j[T_2]^j_i
=\int_{\partial B}|\nabla u|^4 \sigma_2(A) \, dA
+\frac{1}{2}\int_{\partial B}|\nabla u|^4\sigma_2 (A)\, dA
=\frac{3}{2}\int_{\partial B}|\nabla u|^4\sigma_2(A) \, dA.
$$

\noindent Therefore
\begin{align*}
\int_{\partial B} |\nabla u|^{2} u^i u_j[T_2]^j_i (D^2 u) d A
&=\frac{3}{2}\int_{\partial B} |\nabla u|^{4} \sigma_2 (D^2 u) d A.
\end{align*}

\noindent Using these estimates, we obtain
\begin{align*}
I_{2}(\Omega)-I_{2}(B) & \geq \int_{\partial B}\left(\begin{array}{l}
n \\
2
\end{array}\right) \frac{(n-2)(2+1)}{2 n}|\nabla u|^{2}-\left(\begin{array}{c}
n \\
2
\end{array}\right) \frac{(n-2)(2+1)}{2} u^{2} \\
& +\frac{1}{2}|\nabla u|^2\left(-\frac{n(n-1)}{2}-(\frac{3}{2}+O(\epsilon))|\nabla u|^2 \sigma_2(D^2 u)\right)\\
& -\frac{3}{2}\sigma_2(h)^- +(2-\frac{1}{2}n+O(\epsilon)) \Delta udA\\
&+\frac{-3n^2+29n-62}{4}\int_{\partial B} u |\nabla u|^{2} \Delta u d A+O(\epsilon)\|\nabla u\|_{L^{2}}^{2}+O(\epsilon)\|u\|_{L^{2}}^{2}.
\end{align*}

\noindent Note that $\sigma_2(D^2u)=\frac{1}{2}((\Delta u)^2-|\text{Hess}u|^2)$. Thus, it is easy to see that
\begin{align*}
-\int_{\partial B} |\nabla u|^4\sigma_2(D^2 u)dA \geq -\frac{M^2}{2}\int_{\partial B} |\nabla u|^4dA=O(\epsilon)\|\nabla u\|_{L^2}^2,
\end{align*}

\noindent and also
\begin{align*}
\int_{\partial B} u|\nabla u|^2 \Delta u\ dA = O(\epsilon)\|\nabla u\|_{L^2}^2.
\end{align*}

\noindent Combining these two results together, we have
\begin{align*}
I_{2}(\Omega)-I_{2}(B) & \geq \int_{\partial B}\left(\begin{array}{l}
n \\
2
\end{array}\right) \frac{(n-2)(2+1)}{2 n}|\nabla u|^{2}-\left(\begin{array}{c}
n \\
2
\end{array}\right) \frac{(n-2)(2+1)}{2} u^{2} \\
& +\frac{1}{2}|\nabla u|^2\left(-\frac{n(n-1)}{2}+(2-\frac{1}{2}n) \Delta u\right)+O(\epsilon)\sigma_2(h)^-dA\\
& +O(\epsilon)\|\nabla u\|_{L^{2}}^{2}+O(\epsilon)\|u\|_{L^{2}}^{2}.
\end{align*}

\noindent Note that $2-\frac{1}{2}n<0$ and $\Delta u \leq n$. Thus, using Lemma \ref{lem:3.1} with $f(\cdot, \cdot)=1, g(\cdot, \cdot)=1$, we can replace  the term $|\nabla u|^2 \Delta u $ with $|\nabla u_2|^2 \Delta u$ without changing our estimates and get
\begin{align*}
I_{2}(\Omega)-I_{2}(B) & \geq \int_{\partial B} \frac{n^2-4n+3}{2} |\nabla u_1|^2-\frac{3n(n-1)(n-2)}{4}u_1^2\\
&+\frac{n^2-4n+6}{4} |\nabla u_2|^2-\frac{3n(n-1)(n-2)}{4}u_2^2 +O(\epsilon)\sigma_2(h)^- \ dA\\
& +O(\epsilon)\|\nabla u\|_{L^{2}}^{2}+O(\epsilon)\|u\|_{L^{2}}^{2}.
\end{align*}

\noindent For $u_1$, we use the standard Poincare inequality $\displaystyle\int_{\partial B} |\nabla u_1|^2 \geq \int_{\partial B} 2(n+1) u_1^2$ and $n\geq 5$, and for $u_2$, we use the strong Poincare inequality $\displaystyle\int_{\partial B} |\nabla u_2|^2 \geq \int_{\partial B} \lambda (n) u_2^2$, where $\lambda(n)$ is a large eigenvalue that $u_2$ corresponds to. Then we have
\begin{align*}
I_{2}(\Omega)-I_{2}(B) & \geq d(n)\int_{\partial B} |\nabla u|^2+O(\epsilon)\sigma_2(h)^-dA +O(\epsilon)\|\nabla u\|_{L^{2}}^{2}+O(\epsilon)\|u\|_{L^{2}}^{2} .
\end{align*}
\noindent for some positive constant $d(n)>0$.
\end{proof}

\begin{Thm}\label{thm:4.4}
    Suppose $\Omega=\left\{\left(1+u\left(\frac{x}{|x|}\right)\right) x: x \in B\right\} \subseteq \mathbb{R}^{n+1}$ where $u \in C^{3}(\partial B)$, $\operatorname{Vol}(\Omega)=\operatorname{Vol}(B)$, and $\operatorname{bar}(\Omega)=0$. Assume $n \geq 5$, $-M \leq \Delta u \leq n$ for some fixed $M \geq n > 0$. Then, there exists an $\epsilon(n,M)>0$ such that if $\|u\|_{C^1}<\epsilon(n,M)$, we have
\begin{equation}
\int_{\partial \Omega} \sigma_2(h)^+ d\mu-\int_{\partial B} \frac{n(n-1)}{2} dA \geq d(n)\int_{\partial B} |\nabla u|^2 dA
\end{equation}
\noindent for some constant $d(n)>0$, where $\sigma_2^{+}=\max (0, \sigma_2)$.
\end{Thm}

\begin{proof}
    Choose $\epsilon(n,M)$ small enough so that Proposition \ref{s} and the following inequality hold
    \begin{align}\label{eqn: almostequal}
(1-O(\epsilon))\int_{\partial B} \sigma_2(h)^- dA \leq \int_{\partial \Omega} \sigma_2(h)^-d\mu.
    \end{align}
   \noindent  For $\epsilon$ small enough, we have
    \begin{align*}
    &\ \ \int_{\partial \Omega} \sigma_2(h)^+ d\mu -I_2(B)= \int_{\partial \Omega} \sigma_2(h) + \sigma_2(h)^- d\mu-I_2(B) \\&\geq I_{2}(\Omega)-I_{2}(B) + \int_{\partial B}
    (1-O(\epsilon))\sigma_2(h)^- dA \\
    &\geq \tilde{d}(n)\int_{\partial B} |\nabla u|^2+(1-O(\epsilon))\sigma_2(h)^- dA+O(\epsilon)\|\nabla u\|_{L^{2}}^{2}+O(\epsilon)\|u\|_{L^{2}}^{2} \\
    &\geq d(n)\int_{\partial B} |\nabla u|^2dA.    
    \end{align*}
    \noindent In the last line, 
    $0<d(n)<\tilde{d}(n)$.
    We choose $\epsilon$ small enough, so that 
$O(\epsilon)\|\nabla u\|_{L^{2}}^{2}, O(\epsilon)\|u\|_{L^{2}}^{2}$
    are absorbed into $\int_{\partial B} |\nabla u|^2 dA$. 
\end{proof}
\subsection{The $\sigma_k$ case: Proof of Theorem \ref{thm:4.5}}
We can also generalize this situation to the case of $\sigma_k$. Assuming $(-1)^{m} \sigma_m(D^2 u) \geq (-1)^m C_n^m$, we can prove the following theorem. Note that it is easy to verify $\sigma_k(h) \geq0$ in this setting.\\ 
\textbf{Theorem 1.6} \textit{Suppose $\Omega=\left\{\left(1+u\left(\frac{x}{|x|}\right)\right) x: x \in B\right\} \subseteq \mathbb{R}^{n+1}$ where $u \in C^{3}(\partial B)$, $\operatorname{Vol}(\Omega)=\operatorname{Vol}(B)$, and $\operatorname{bar}(\Omega)=0$. Assume $n \geq 5$, $\|u\|_{C^2}\leq  M$ for some constant $M$, $(-1)^{m} \sigma_m(D^2 u) \geq (-1)^m C_n^m$ for integers $1\le m \le k$. Then, for $k<n-1$, there exists an $\epsilon(k,n,M)>0$ such that if $\|u\|_{C^1}<\epsilon(k,n,M)$, we have
\begin{equation}
\int_{\partial \Omega} \sigma_k(h) d\mu-\int_{\partial B} C_n^k dA \geq d(k,n)\int_{\partial B} |\nabla u|^2 dA
\end{equation}
\noindent for some constant $d(k,n)>0$.}
%\end{Thm*}
\begin{rem}
In contrast to the statement of Theorem \ref{thm:4.4} in the previous subsection, we make stronger assumptions in Theorem \ref{thm:4.5}, and the inequality is also stronger. 
\end{rem}
\noindent \begin{p} Under the assumption that $\operatorname{Vol}(\Omega)=\operatorname{Vol}(B)$ and $\operatorname{bar}(\Omega)=0$, using the result \cite[Proposition 4.3]{VW22}, which can be proven under the condition that the $C^{1}$-norm of $u$ is small and the $C^2$ and $C^3$-norms of $u$ are bounded, we obtain
	
\begin{equation} \label{Eq: 6.11}
\begin{aligned}
I_k(\Omega) - I_k(B)& = \int_{\partial B}  C_n^k \cdot \frac{(n-k)(k+1)}{2 n} |\nabla u|^2 - C_n^k \cdot \frac{(n-k)(k+1)}{2} u^2 \\
& + \sum_{m=1}^k (-1)^m C_{n-m}^{k-m} \frac{(n-k)(k+1)}{2(m+1) (n-m)} |\nabla u|^2 \sigma_m\left(D^2 u\right)dA \\
& + O(\epsilon) \|\nabla u\|^2_{L^2} + O(\epsilon) \|u\|^2_{L^2}.
\end{aligned}
\end{equation}
\noindent Then using the condition $(-1)^{m} \sigma_m(D^2 u) \geq (-1)^m C_n^m$, we obtain that 
\begin{equation} \label{Eq: 6.12}
\begin{aligned}
I_k(\Omega) - I_k(B)& = \int_{\partial B}  C_n^k \cdot \frac{(n-k)(k+1)}{2 n} |\nabla u|^2 - C_n^k \cdot \frac{(n-k)(k+1)}{2} u^2 \\
& + \sum_{m=1}^k (-1)^m C_{n-m}^{k-m} \frac{(n-k)(k+1)}{2(m+1) (n-m)} |\nabla u|^2 \sigma_m\left(D^2 u\right)dA \\
& + O(\epsilon) \|\nabla u\|^2_{L^2} + O(\epsilon) \|u\|^2_{L^2}.\\\
& \geq \int_{\partial B}  C_n^k \cdot \frac{(n-k)(k+1)}{2 n} |\nabla u|^2 - C_n^k \cdot \frac{(n-k)(k+1)}{2} u^2 \\
& + \sum_{m=1}^k(-1)^m C_{n-m}^{k-m} \frac{(n-k)(k+1)}{2(m+1) (n-m)} |\nabla u|^2 C_n^m dA\\
& + O(\epsilon) \|\nabla u\|^2_{L^2} + O(\epsilon) \|u\|^2_{L^2}.
\end{aligned}
\end{equation}

\noindent Then, by applying Lemma \ref{lem:5.2}, we see that the coefficient is nonnegative and strictly positive when $k < n - 1$. However, this coefficient is not sufficiently large to establish the desired inequality after applying the Poincaré inequality. Indeed, since we can only assume that $u$ is orthogonal to the first $n$ eigenfunctions of the Laplacian, we have $\displaystyle\int_{\partial B} |\nabla u|^2 dA \geq 2(n+1) \int_{\partial B} |u|^2 dA $, which is not enough to ensure that the final estimate is nonnegative. Therefore, in order to prove the stability—or even the validity of the inequality—a more refined argument is required.

To overcome this issue, we decompose $u=u_{1}+u_{2}$ where $u_{1}$ is the low-frequency component of $u$ and $u_{2}$ is the high-frequency component (or equivalently $u_{1}$ is the projection of $u$ on the subspace of the eigenfunctions of $-\Delta$ with eigenvalues smaller than $\lambda=\lambda(n)$ ). Since we can prove that only $u_{2}$ matters in the term $|\nabla u|^2\sigma_1(D^2u)$, we deduce that  the coefficient for $|\nabla u_1|^2$ is given by 

\begin{equation*}
\sum_{m=0}^k (-1)^m C_{n-m}^{k-m} \frac{(n-k)(k+1)}{2(m+1)(n-m)} \cdot C_{n}^{m} + C_{n-1}^{k-1} \cdot \frac{(n-k)(k+1)}{2 \cdot (1+1) \cdot (n-1)} \cdot C_{n}^{1},
\end{equation*}

\noindent where the additional term 

\begin{equation*}
C_{n-1}^{k-1} \cdot \frac{(n-k)(k+1)}{2 \cdot (1+1) \cdot (n-1)} \cdot C_{n}^{1}
\end{equation*}

\noindent appears because we substitute $|\nabla u_2|^2 \sigma_1(D^2 u)$ in place of $|\nabla u|^2 \sigma_1(D^2 u)$ in the last estimate. This coefficient is sufficiently large to apply the standard Poincaré inequality. The coefficient for $|\nabla u_2|^2$ is positive, as  

\begin{equation}
\sum_{m=0}^k (-1)^m C_{n-m}^{k-m} \frac{(n-k)(k+1)}{2(m+1)(n-m)} \cdot C_{n}^{m}>0,
\end{equation}
whose proof is given in Lemma \ref{lem:5.2}.
%\edz{This inequality has not been proved. Both $u_1$ and $u_2$.}
This positivity enables the application of the stronger Poincaré inequality since $u_{2}$ enjoys a very strong Poincaré inequality.

Using the argument above and noting that $\|u\|_{C^1} \leq \varepsilon$, we obtain

\begin{equation*}
\begin{aligned}
    &\quad \ \  I_k(\Omega) - I_k(B) \\
    &\geq\int_{\partial B}\{\sum_{m=0}^k (-1)^m C_{n-m}^{k-m} \frac{(n-k)(k+1)}{2(m+1)(n-m)} \cdot C_{n}^{m} + C_{n-1}^{k-1} \cdot \frac{(n-k)(k+1)}{2 \cdot (1+1) \cdot (n-1)} \cdot C_{n}^{1} \} |\nabla u_1|^2 \\
    &-C_n^k \frac{(n-k)(k+1)}{2} u_1^2 + \left\{\sum_{m=0}^k (-1)^m C_{n-m}^{k-m} \frac{(n-k)(k+1)}{2(m+1)(n-m)} \cdot C_{n}^{m}\right\} |\nabla u_2|^2 \\
    &- C_n^k \frac{(n-k)(k+1)}{2} u_2^2  dA +O(\varepsilon) \|\nabla u\|_{L^2}^2,
\end{aligned}
\end{equation*}

\noindent for some constant $O(\varepsilon)$. 
Since 

$$
C_{n-1}^{k-1} \cdot \frac{(n-k)(k+1)}{2 \cdot (1+1) \cdot (n-1)} \cdot C_{n}^{1} \cdot (2n+2) > C_n^k \frac{(n-k)(k+1)}{2},
$$ 

\noindent we can apply the regular Poincaré inequality  
$\displaystyle
\int_{\partial B} |\nabla u|^2 dA
\geq 2(n+1)\int_{\partial B} |u|^2 dA
$
to $u_1$.

For $u_2$, we can apply a stronger version of the Poincaré inequality 
$\displaystyle
\int_{\partial B} |\nabla u|^2 dA
\geq \lambda(n)\int_{\partial B} |u|^2 dA
$, where $\lambda(n)$ is a large eigenvalue that $u_2$ corresponds to. This will ensure that

\begin{equation*}
    \left\{\sum_{m=0}^k (-1)^m C_{n-m}^{k-m} \frac{(n-k)(k+1)}{2(m+1)(n-m)} \cdot C_{n}^{m}\right\} |\nabla u_2|^2 - C_n^k \frac{(n-k)(k+1)}{2} u_2^2 \geq d(k,n) |\nabla u_2|^2,
\end{equation*}

\noindent for some constant $d(k, n) > 0$. Thus, the result follows by choosing $\varepsilon$ sufficiently small.
%At this point, we can use formula \eqref{eq:leading} to derive that up to an $O(\epsilon^2)$ term (???), $\sigma_k(h)\geq 0$. Thus 
% $$
% \int_{\partial\Omega} \sigma_k(h) d\mu= 
% \int_{\partial\Omega} \sigma_k(h)^- d\mu+ O(\epsilon^2),$$ which completes the proof. 

\end{p}

\section{Proof of Theorem \ref{thm:5.1}}

 We want to obtain a result similar to the one achieved by F. Glaudo \cite[Theorem 1.1]{GLAUDO2022108595} regarding the Minkowski inequality. Before proving this theorem, we first establish some preliminary results.
 Under the assumptions that $\operatorname{Vol}(\Omega)=\operatorname{Vol}(B)$ and $\operatorname{bar}(\Omega)=0$, using the result \cite[Proposition 4.3]{VW22}, which can be proved under the condition that the $C^{1}$-norm of $u$ is small and the $C^2$ and $C^3$ norms of $u$ are bounded, we obtain
	
	\begin{equation} \label{Eq: 6.1}
	\begin{aligned}
		I_k(\Omega) - I_k(B)& = \int_{\partial B}  C_n^k \cdot \frac{(n-k)(k+1)}{2 n} |\nabla u|^2 - C_n^k \cdot \frac{(n-k)(k+1)}{2} u^2 \\
		& + \sum_{m=1}^k (-1)^m C_{n-m}^{k-m} \frac{(n-k)(k+1)}{2(m+1) (n-m)} |\nabla u|^2 \sigma_m\left(D^2 u\right)dA \\
		& + O(\epsilon) \|\nabla u\|^2_{L^2} + O(\epsilon) \|u\|^2_{L^2}.
	\end{aligned}
	\end{equation}
	
	\noindent Next, we substitute $\sigma_j(h)$ for $\sigma_j\left(D^2 u\right)$. Using the expression
    \cite[Lemma 3.1]{VW22}, we get
    
	\begin{equation}\label{eq:leading}
	\begin{aligned}
		\sigma_k(h) &= \frac{1}{((1+u)^2 + |\nabla u|^2)^{\frac{k+2}{2}}} \cdot \sum_{m=0}^k (-1)^m \cdot C_{n-m}^{k-m} \cdot \frac{(1+u)^2 \sigma_m(D^2 u)}{(1+u)^m} + O(\epsilon) \\
		& = \sum_{m=0}^k   (-1)^m \cdot C_{n-m}^{k-m} \cdot \sigma_m\left(D^2 u\right) + O(\epsilon).
	\end{aligned}
	\end{equation}
\begin{Lem}\label{lem:binomial-involution}
For $0\le k\le n$,
\begin{equation}\label{Eq:6.3}
\sigma_k(D^2 u)
=\sum_{m=0}^k (-1)^m C_{n-m}^{k-m}\,\sigma_m(h) \;+\; O(\varepsilon).
\end{equation}
\end{Lem}

\begin{proof}
Ignoring the $O(\varepsilon)$ terms, the relation between the elementary symmetric functions
of $D^2u$ and those of $h$ can be written in matrix form as
$$
\Sigma^{(h)} = A\,\Sigma^{(u)}, \qquad 
A_{km} = 
\begin{cases}
(-1)^m C_{n-m}^{k-m}, & k\ge m,\\
0, & k<m,
\end{cases}
$$
where $\Sigma^{(u)}=(\sigma_0(D^2u),\dots,\sigma_n(D^2u))^\top$ and 
$\Sigma^{(h)}=(\sigma_0(h),\dots,\sigma_n(h))^\top$.  
The matrix $A$ is lower triangular with diagonal entries equal to $1$.  
Hence the desired formula~\eqref{Eq:6.3} is equivalent to showing that $A^2=\mathrm{Id}$.

To verify this, we compute
$$
(A^2)_{k\ell}
=\sum_{m=\ell}^k A_{km}A_{m\ell}
=\sum_{m=\ell}^k (-1)^{m+\ell}C_{n-m}^{k-m}C_{n-\ell}^{m-\ell}.
$$
A simple reindexing of the summation (set $i=\ell+1$, $j=k+1$, $h=m+1$) transforms the right-hand side into
$$
(A^2)_{k\ell}
=(-1)^{i-1}\sum_{h=i}^j (-1)^{h-1}
C_{(n+1)-h}^{\,j-h\,}C_{(n+1)-i}^{\,h-i\,}.
$$
We claim that, for integers $1\le i\le j\le n+1$,
$$
\sum_{h=i}^j (-1)^{h-i}C_{(n+1)-h}^{\,j-h\,}C_{(n+1)-i}^{\,h-i\,}=\delta_{ij}.
$$
Indeed, writing $R=n+1-i$, $r=j-i$, and $m=h-i$, the sum becomes
$$
S=\sum_{m=0}^{r}(-1)^m C_{R-m}^{\,r-m\,}C_{R}^{\,m\,}.
$$
Using the elementary combinatorial identity 
$$
C_{R}^{m}C_{R-m}^{R-r}=C_{R}^{R-r}C_{r}^{m},
$$
and noting $C_{R-m}^{r-m}=C_{R-m}^{R-r}$, we obtain
$$
S=C_{R}^{R-r}\sum_{m=0}^r (-1)^m C_{r}^{m}
=C_{R}^{R-r}(1-1)^r.
$$
Thus $S=0$ if $r>0$ and $S=1$ if $r=0$, proving the identity above.  
Consequently $(A^2)_{k\ell}=\delta_{k\ell}$ for all $0\le \ell\le k\le n$, and hence $A^2=\mathrm{Id}$.

Finally, applying $A$ to both sides of $\Sigma^{(h)}=A\Sigma^{(u)}+O(\varepsilon)$ yields
$$
\Sigma^{(u)}=A\Sigma^{(h)}+O(\varepsilon),
$$
which is precisely~\eqref{Eq:6.3}.
\end{proof}
\noindent Substituting this expression into the expression of $I_k(\Omega) - I_k(B)$, we get

\begin{equation}
\begin{aligned}
	I_k(\Omega) - I_k(B)&= \int_{\partial B} C_n^k \frac{(n-k)(k+1)}{2n} |\nabla u|^2 
	- C_n^k \frac{(n-k)(k+1)}{2} u^2 \\
	& + \sum_{m=1}^k (-1)^m C_{n-m}^{k-m} \frac{(n-k)(k+1)}{2(m+1)(n-m)} |\nabla u|^2 \cdot \left\{ \sum_{j=0}^m (-1)^j C_{n-j}^{m-j} \, \sigma_j\left(h\right) \right\}dA\\
    &+ O(\varepsilon) \|\nabla u\|_{L^2}^2 + O(\varepsilon) \|u\|_{L^2}^2.
\end{aligned}
\end{equation}

	\noindent Next, we can decide the signs of $\sigma_m(h)$ to make $I_k(\Omega)-I_k(B)$ nonnegative. This is equivalent to decide the sign of $\sum_{m=j}^k \frac{1}{(m-j)!(k-m)!} \cdot \frac{1}{(m+1)(n-m)} \cdot (-1)^{j+m}$. We proceed with the computation for $k < n$.

\noindent Note that 

\begin{equation*}
\sum_{m=0}^k (-1)^m C_{n-m}^{k-m} \frac{(n-k)(k+1)}{2(m+1)(n-m)} \cdot C_{n}^{m}
\end{equation*}

\noindent is the coefficient associated with $|\nabla u|^2$, and it's positive by Lemma \ref{lem:5.2}.

 \noindent Using Lemma \ref{lem:3.1}, we can replace the term $|\nabla u|^2 \Delta u$ by $|\nabla u_2|^2 \Delta u$ without affecting the proof argument, where $u_1$ and $u_2$ are the decomposition components of $u$ as previously defined. After this replacement, we examine the coefficients associated with $|\nabla u_1|^2$ and $|\nabla u_2|^2$. 

\noindent The coefficient for $|\nabla u_2|^2$ is positive, as
\begin{equation*}
\sum_{m=0}^k (-1)^m C_{n-m}^{k-m} \frac{(n-k)(k+1)}{2(m+1)(n-m)} \cdot C_{n}^{m}>0.
\end{equation*}This positivity enables us to apply the stronger Poincaré inequality. Meanwhile, the coefficient for $|\nabla u_1|^2$ is given by 

\begin{equation*}
\sum_{m=0}^k (-1)^m C_{n-m}^{k-m} \frac{(n-k)(k+1)}{2(m+1)(n-m)} \cdot C_{n}^{m} + C_{n-1}^{k-1} \cdot \frac{(n-k)(k+1)}{2 \cdot (1+1) \cdot (n-1)} \cdot C_{n}^{1},
\end{equation*}

\noindent where the additional term 

\begin{equation*}
C_{n-1}^{k-1} \cdot \frac{(n-k)(k+1)}{2 \cdot (1+1) \cdot (n-1)} \cdot C_{n}^{1}
\end{equation*}

\noindent appears because we replace $|\nabla u|^2 \sigma_1(D^2 u)$ by $|\nabla u_2|^2 \sigma_1(D^2 u)$ before incorporating (\ref{Eq:6.3}) into (\ref{Eq: 6.1}). This coefficient is sufficiently large to apply the standard Poincaré inequality. Now we can start to prove Theorem \ref{thm:5.1}.

\noindent \begin{p}of Theorem \ref{thm:5.1}. Using the argument above and noting that $\|u\|_{C^1} \leq \varepsilon$, we obtain

\begin{equation*}
\begin{aligned}
    &\quad \ \  I_k(\Omega) - I_k(B) \\
    &\geq\int_{\partial B}\{\sum_{m=0}^k (-1)^m C_{n-m}^{k-m} \frac{(n-k)(k+1)}{2(m+1)(n-m)} \cdot C_{n}^{m} + C_{n-1}^{k-1} \cdot \frac{(n-k)(k+1)}{2 \cdot (1+1) \cdot (n-1)} \cdot C_{n}^{1} \} |\nabla u_1|^2 \\
    &-C_n^k \frac{(n-k)(k+1)}{2} u_1^2 + \left\{\sum_{m=0}^k (-1)^m C_{n-m}^{k-m} \frac{(n-k)(k+1)}{2(m+1)(n-m)} \cdot C_{n}^{m}\right\} |\nabla u_2|^2 \\
    &- C_n^k \frac{(n-k)(k+1)}{2} u_2^2 + O_1(\varepsilon) \cdot \left\{\sum_{j=1}^{k} \sigma_j(h)^-\right\}dA+ O(\varepsilon) \|\nabla u\|_{L^2}^2 + O(\varepsilon) \|u\|_{L^2}^2,
\end{aligned}
\end{equation*}

\noindent for some constants $O_1(\varepsilon)$ and $O(\varepsilon)$. 

Since 

$$
C_{n-1}^{k-1} \cdot \frac{(n-k)(k+1)}{2 \cdot (1+1) \cdot (n-1)} \cdot C_{n}^{1} \cdot (2n+2) > C_n^k \frac{(n-k)(k+1)}{2},
$$ 

\noindent we can apply the regular Poincaré inequality  
$\displaystyle
\int_{\partial B} |\nabla u|^2 dA
\geq 2(n+1)\int_{\partial B} |u|^2 dA
$
to $u_1$.

For $u_2$, we can apply a stronger version of the Poincaré inequality 
$\displaystyle
\int_{\partial B} |\nabla u|^2 dA
\geq \lambda(n)\int_{\partial B} |u|^2 dA
$, where $\lambda(n)$ is a large eigenvalue that $u_2$ corresponds to. This will ensure that

\begin{equation*}
    \left\{\sum_{m=0}^k (-1)^m C_{n-m}^{k-m} \frac{(n-k)(k+1)}{2(m+1)(n-m)} \cdot C_{n}^{m}\right\} |\nabla u_2|^2 - C_n^k \frac{(n-k)(k+1)}{2} u_2^2 \geq d(k, n) |\nabla u_2|^2,
\end{equation*}

\noindent for some constant $d(k, n) > 0$. Thus, the result follows by choosing $\varepsilon$ sufficiently small.

\end{p}
\section{The Proof of Theorem \ref{thm:5.2}}
If we assume
\begin{equation}
\int_{\partial \Omega} \sigma_{j'}(h)d\mu=\int_{\partial B}C_n^{j'}dA
\end{equation}

\noindent then from 
Lemma 4.2 in \cite{VW22}
we get
\begin{equation} 
	\begin{aligned}
		&\int_{\partial B}  C_n^{j'} \cdot (n-j')u+C_n^{j'}\frac{(n-j')(n-j'-1)}{2}u^2+ \sum_{m=0}^
        {j'}(-1)^m C_{n-m}^{j'-m} \frac{(n-j')(j'+1)}{2(m+1) (n-m)} |\nabla u|^2 \sigma_m\left(D^2 u\right)dA\\
		&= -O(\epsilon) \|\nabla u\|^2_{L^2} + O(\epsilon) \|u\|^2_{L^2}.
	\end{aligned}
\end{equation}

\noindent Therefore
\begin{equation} 
	\begin{aligned}\label{j'}
		&\int_{\partial B}  u+\frac{(n-j'-1)}{2}u^2+ \sum_{m=0}^
        {j'}(-1)^m C_{n-m}^{j'-m} \frac{(j'+1)}{2(m+1) (n-m)C_n^{j'}} |\nabla u|^2 \sigma_m\left(D^2 u\right)dA\\
		&= -O(\epsilon) \|\nabla u\|^2_{L^2} + O(\epsilon) \|u\|^2_{L^2}.
	\end{aligned}
\end{equation}

\noindent We know that 
\begin{equation}
\begin{aligned}
&\int_{\partial \Omega} \sigma_{k}(h)d\mu-\int_{\partial B}C_n^{k}dA\\
=&\int_{\partial B}  C_n^{k} \cdot (n-k)u+C_n^{k}\frac{(n-k)(n-k-1)}{2}u^2+ \sum_{m=0}^
        {k}(-1)^m C_{n-m}^{k-m} \frac{(n-k)(k+1)}{2(m+1) (n-m)} |\nabla u|^2 \sigma_m\left(D^2 u\right)dA\\
        +&O(\epsilon) \|\nabla u\|^2_{L^2} + O(\epsilon) \|u\|^2_{L^2}.
	\end{aligned}
\end{equation}

\noindent Substituting (\ref{j'}) into the equation above and using Lemma \ref{lem:5.2} and Lemma \ref{lem:binomial-involution}, we get that 
\begin{equation}
\begin{aligned}
&\int_{\partial \Omega} \sigma_{k}(h)\,d\mu
    -\int_{\partial B} C_n^{k}\,dA \\[0.6em]
%
% -------- FIRST EXPANSION ------------
%=&\int_{\partial B}\Bigg[    C_n^{k}(n-k)\,u   +C_n^{k}\frac{(n-k)(n-k-1)}{2}\,u^2\\&\qquad\qquad +\sum_{m=0}^{k} (-1)^m C_{n-m}^{k-m}        \frac{(n-k)(k+1)}{2(m+1)(n-m)} |\nabla u|^2\,\sigma_m(D^2u)\Bigg]\,dA \\[0.6em]\quad+&O(\varepsilon)\|\nabla u\|_{L^2}^2+O(\varepsilon)\|u\|_{L^2}^2 \\[1.0em]
%
% -------- SECOND EXPANSION (SUBSTITUTING THE LEMMA) ------------
=&\int_{\partial B}\Bigg[
    C_n^{k}(n-k)\,u
    +C_n^{k}\frac{(n-k)(n-k-1)}{2}\,u^2
\\
+&\sum_{\ell=0}^{k}
        \Bigg(
            \sum_{m=\ell}^{k}
            (-1)^{m+\ell}
            C_{n-m}^{k-m}
            C_{n-\ell}^{m-\ell}
            \frac{(n-k)(k+1)}{2(m+1)(n-m)}
        \Bigg)
        |\nabla u|^2\,\sigma_\ell(h)
\Bigg]\,dA \\[0.6em]
+&O(\varepsilon)\|\nabla u\|_{L^2}^2
+O(\varepsilon)\|u\|_{L^2}^2\\
=&\int_{\partial B}\Bigg[
    C_n^{k}(n-k)\,u
    +C_n^{k}\frac{(n-k)(n-k-1)}{2}\,u^2
\\
%&\qquad\qquad
+&\sum_{\ell=0}^{k}
        \Bigg(\frac{(n-l)!(n-k)(k+1)}{2(n-k)!}
            \sum_{m=l}^k \frac{1}{(m-l)!(k-m)!} \cdot \frac{1}{(m+1)(n-m)} \cdot (-1)^{l+m}
        \Bigg)
        |\nabla u|^2\,\sigma_\ell(h)
\Bigg]\,dA \\[0.6em]
+&O(\varepsilon)\|\nabla u\|_{L^2}^2
+O(\varepsilon)\|u\|_{L^2}^2.\\
%=&\int_{\partial B}\Bigg[   C_n^{k}(n-k)\,u    +C_n^{k}\frac{(n-k)(n-k-1)}{2}\,u^2\\
%&\qquad\qquad    +\sum_{\ell=0}^{k}\Bigg(&\frac{(n-l)!(n-k)(k+1)}{2(n-k)!}            \{\frac{1}{(k-l)!(n+1)}\{\frac{l!(k-l)!}{(k+1)!}+(-1)^{k-l}\frac{(n-k-1)!(k-l)!}{(n-l)!}\}\}\Bigg)|\nabla u|^2\,\sigma_\ell(h)\Bigg]\,dA \\[0.6em]
=&\int_{\partial B}  C_n^{k}\frac{(n-k)(j'-k)}{2}u^2\\
+&\sum_{m=0}^{j'} C_n^k (n-k) \{ \frac{1}{2(n+1)C_n^m}+\frac{(-1)^{k-m} }{2C_n^{k+1}(n+1)}-\frac{1}{2(n+1) C_n^m}-\frac{(-1)^{j'-m} }{2(n+1)C_n^{j'+1}} \}|\nabla u|^2\sigma_m(h)\\
+&\sum_{m=j'+1}^{k}\frac{k+1}{2(n+1)} \{\frac{C_n^{k+1}}{C_n^m}+(-1)^{k-m}\} |\nabla u|^2 \sigma_m (h)dA\\
        +&O(\epsilon) \|\nabla u\|^2_{L^2} + O(\epsilon) \|u\|^2_{L^2}.
	\end{aligned}
\end{equation}

\noindent So we get
\begin{equation}\label{222}
\begin{aligned}
&\int_{\partial \Omega} \sigma_{k}(h)d\mu-\int_{\partial B}C_n^{k}dA\\
=&\int_{\partial B}  C_n^{k}\frac{(n-k)(j'-k)}{2}u^2+\sum_{m=0}^{j'} (-1)^m\frac{C_n^k (n-k)}{2(n+1)} \{ (-1)^{k}\frac{1}{C_n^{k+1}}-(-1)^{j'}\frac{1}{C_n^{j'+1}} \}|\nabla u|^2\sigma_m(h)\\
+&\sum_{m=j'+1}^{k}\frac{k+1}{2(n+1)} \{\frac{C_n^{k+1}}{C_n^m}+(-1)^{k-m}\} |\nabla u|^2 \sigma_m (h)dA\\
        +&O(\epsilon) \|\nabla u\|^2_{L^2} + O(\epsilon) \|u\|^2_{L^2}.
	\end{aligned}
\end{equation}

\noindent If we assume that $j'+1+k+1>n$, then we have $\frac{C_n^{k+1}}{C_n^m}<1$ and $\frac{1}{C_n^{k+1}}-\frac{1}{C_n^{j'+1}}>0$, so if we further assume that $j'$ is even, then we get that the coefficient of $|\nabla u|^2$ is positive, and if we denote $\frac{C_n^k (n-k)}{2(n+1)} \{ (-1)^{k}\frac{1}{C_n^{k+1}}-(-1)^{j'}\frac{1}{C_n^{j'+1}} \}$ by $A>0$, then the above equation can be further simplified to 
\begin{equation}
\begin{aligned}
&\int_{\partial \Omega} \sigma_{k}(h)d\mu-\int_{\partial B}C_n^{k}dA\\
=&\int_{\partial B}  C_n^{k}\frac{(n-k)(j'-k)}{2}u^2+A|\nabla u|^2+\sum_{m=1}^{j'} (-1)^mA|\nabla u|^2\sigma_m(h)\\
+&\sum_{m=j'+1}^{k}\frac{k+1}{2(n+1)} \{\frac{C_n^{k+1}}{C_n^m}+(-1)^{k-m}\} |\nabla u|^2 \sigma_m (h)dA\\
        +&O(\epsilon) \|\nabla u\|^2_{L^2} + O(\epsilon) \|u\|^2_{L^2}.\\
	\end{aligned}
\end{equation}

\noindent As before, by Lemma \ref{lem:3.1}, we may replace the term $|\nabla u|^2 \Delta u$ with $|\nabla u_2|^2 \Delta u$ without affecting the proof argument, where $u_1$ and $u_2$ are the decomposition components of $u$ as previously defined. After this replacement, we examine the coefficients associated with $|\nabla u_1|^2$ and $|\nabla u_2|^2$. The coefficient for $|\nabla u_2|^2$ is positive, as $A$ is positive. This positivity enables the application of the stronger Poincaré inequality. Meanwhile, the coefficient for $|\nabla u_1|^2$ is given by 

\begin{equation*}
A + C_{n-1}^{k-1} \cdot \frac{(n-k)(k+1)}{2 \cdot (1+1) \cdot (n-1)} \cdot C_{n}^{1},
\end{equation*}

\noindent where the additional term 

\begin{equation*}
C_{n-1}^{k-1} \cdot \frac{(n-k)(k+1)}{2 \cdot (1+1) \cdot (n-1)} \cdot C_{n}^{1}
\end{equation*}

\noindent appears as before. This coefficient is sufficiently large to apply the standard Poincaré inequality.  Now we can start to prove Theorem \ref{thm:5.2}.

\noindent \begin{p}of Theorem \ref{thm:5.2}. Using the argument above and noting that $\|u\|_{C^1} \leq \varepsilon$, we obtain

\begin{equation*}
\begin{aligned}
    &\quad \ \  I_k(\Omega) - I_k(B) \\
    &\geq\int_{\partial B}\{A + C_{n-1}^{k-1} \cdot \frac{(n-k)(k+1)}{2 \cdot (1+1) \cdot (n-1)} \cdot C_{n}^{1} \} |\nabla u_1|^2 \\
    &-C_n^k \frac{(n-k)(k-j')}{2} u_1^2 + A |\nabla u_2|^2 \\
    &- C_n^k \frac{(n-k)(k-j')}{2} u_2^2 + A \cdot \left\{\sum_{m=1}^{j'} (-1)^m \sigma_m(h)\right\}|\nabla u|^2\\
    &+\sum_{m=j'+1}^{k}\frac{k+1}{2(n+1)} \{\frac{C_n^{k+1}}{C_n^m}+(-1)^{k-m}\} |\nabla u|^2 \sigma_m (h)dA\\
    &+O(\epsilon) \|\nabla u\|^2_{L^2} + O(\epsilon) \|u\|^2_{L^2}
\end{aligned}
\end{equation*}

\noindent for some constants $O_1(\varepsilon)$ and $O(\varepsilon)$. 

Since 

$$
C_{n-1}^{k-1} \cdot \frac{(n-k)(k+1)}{2 \cdot (1+1) \cdot (n-1)} \cdot C_{n}^{1} \cdot (2n+2) > C_n^k \frac{(n-k)(k-j')}{2},
$$ 

\noindent we can apply the regular Poincaré inequality  
$\displaystyle
\int_{\partial B} |\nabla u|^2 dA
\geq 2(n+1)\int_{\partial B} |u|^2 dA
$
to $u_1$.

For $u_2$, we can apply a stronger version of the Poincaré inequality 
$
\int_{\partial B} |\nabla u|^2 dA
\geq \lambda(n)\int_{\partial B} |u|^2 dA
$, where $\lambda(n)$ is a large eigenvalue that $u_2$ corresponds to. This will ensure that

\begin{equation*}
    A |\nabla u_2|^2 - C_n^k \frac{(n-k)(k-j')}{2} u_2^2 \geq c(n) |\nabla u_2|^2,
\end{equation*}

\noindent for some constant $c(n) > 0$. By the condition, we have $\sum_{m=1}^{j'} (-1)^m \sigma_m(h) \geq 0$ and $\sum_{m=j'+1}^{k} \{\frac{C_n^{k+1}}{C_n^m}+(-1)^{k-m}\} |\nabla u|^2 \sigma_m (h)\geq 0$.  Thus, the result follows by choosing $\varepsilon$ sufficiently small.
\end{p}
\section{The Proof of Theorem \ref{thm:3}}
We repeat the proof of Theorem \ref{thm:5.2} verbatim up to (\ref{222}).  
As a consequence of (\ref{222}), substituting $j'$ with $0$, we have 
\begin{equation}
\begin{aligned}
&\int_{\partial \Omega} \sigma_{k}(h)d\mu-\int_{\partial B}C_{n}^{k}dA\\
=&\int_{\partial B}  C_{n}^{k}\frac{(n-k)(-k)}{2}u^2
+\frac{C_{n}^{k} (n-k)}{2(n+1)} 
\left\{ (-1)^{k}\frac{1}{C_{n}^{k+1}}-\frac{1}{C_{n}^{1}} \right\}
|\nabla u|^2\\
+&\sum_{m=1}^{k}\frac{k+1}{2(n+1)} 
\left\{\frac{C_{n}^{k+1}}{C_{n}^{m}}+(-1)^{k-m}\right\} 
|\nabla u|^2 \sigma_m (h)dA\\
+&O(\epsilon) \|\nabla u\|^2_{L^2} + O(\epsilon) \|u\|^2_{L^2}.
\end{aligned}
\end{equation}
%\edz{In the second line of (7.1), it is $m=0$. Also, is $k$ even in Theorem 1.4? If not, why in the second line of (7.1), coefficient of $|\nabla u|^2$ is positive?"nonnegative" should be $>0$, right?}
\noindent Now since $k<\left\lfloor \frac{n}{2}\right\rfloor$, we have 
$\dfrac{C_{n}^{k+1}}{C_{n}^{m}}+(-1)^{k-m}\geq 0$, %so the coefficient of $|\nabla u|^2\sigma_m(h_i^j)$ is nonnegative 
for every $m$ between $1$ and $k$; 
and $$\int_{\partial B}
\frac{C_{n}^{k} (n-k)}{2(n+1)} 
\left\{ (-1)^{k}\frac{1}{C_{n}^{k+1}}-\frac{1}{C_{n}^{1}} \right\}
|\nabla u|^2 dA\geq \omega(\epsilon).$$

\noindent Therefore, we obtain that 
\begin{equation}
\int_{\partial \Omega}  \sigma_k(h)+ \sum_{j=1}^{k} \sigma_j(h)^-\, d\mu
-\int_{\partial B} C_{n}^{k}\, dA 
\geq \omega(\epsilon)
+(1-\omega(\epsilon))\int_{\partial \Omega} \sum_{j=1}^{k} \sigma_j(h)^-d\mu
\geq \omega(\epsilon),
\end{equation}

\noindent which implies the desired statement since $|\omega(\epsilon)|$ can be assumed to be much smaller than the constant $\delta$ appearing in the statement.

 \section{The Proof of Theorem \ref{thm:5.3}}
 Here, $\omega(\varepsilon)$ denotes a quantity that tends to $0$ as $\varepsilon \to 0$.\\
\noindent Recall that
\begin{equation}
\begin{aligned}\label{j}
I_{k}(\Omega)-I_{k}(B) & =\int_{\partial B}C_{n}^{k}(n-k) u+C_{n}^{k}\frac{(n-k)(n-k-1)}{2} u^{2} \\
& +\sum_{m=0}^{k}(-1)^{m}C_{n-m}^{k-m} \frac{(n-k)(k+1)}{2(m+1)(n-m)}|\nabla u|^{2} \sigma_{k}\left(D^{2} u\right) d A+O(\epsilon)\|\nabla u\|_{L^{2}}^{2}+O(\epsilon)\|u\|_{L^{2}}^{2}.
\end{aligned}
\end{equation}

\noindent Therefore, as before, if $I_{j}(\Omega)=I_{j}(B)$,

\begin{equation}
\begin{aligned}
&\int_{\partial B} u d A\\
 =&-\int_{\partial B} \frac{n-j-1}{2} u^{2}+\frac{j+1}{2 n}|\nabla u|^{2}+\left(\frac{1}{C_{n}^{j}} \sum_{m=1}^j (-1)^{m}C_{n-m}^{j-m} \frac{j+1}{2(m+1)(n-m)} |\nabla u|^2\sigma_{m}\left(D^{2} u\right)\right) d A \\
& +O(\epsilon)\|\nabla u\|_{L^{2}}^{2}+O(\epsilon)\|u\|_{L^{2}}^{2}. 
\end{aligned}
\end{equation}

\noindent Substituting this expression into (\ref{j}) yields

\begin{equation}
\begin{aligned}\label{k}
I_{k}(\Omega)-I_{k}(B)  =&C_{n}^{k} \frac{(n-k)(k-j)}{2 n} \int_{\partial B}|\nabla u|^{2}-n u^{2}+\left(\sum_{m=1}^{k} d_{m} \sigma_{m}\left(D^{2} u\right)\right)|\nabla u|^{2} d A \\
& +O(\epsilon)\|\nabla u\|_{L^{2}}^{2}+O(\epsilon)\|u\|_{L^{2}}^{2}, 
\end{aligned}
\end{equation}

\noindent where each $d_{m}$ is the coefficient for $\sigma_{m}\left(D^{2} u\right)$.\\

\noindent As a special case, under the condition $I_0(\Omega)=I_0(B)$, we have 
\begin{equation}
\begin{aligned}\label{k2}
&I_{k}(\Omega)-I_{k}(B)\\
 =&C_{n}^{k} \frac{(n-k)k}{2 n} \int_{\partial B}|\nabla u|^{2}-n u^{2}dA+ \int_{\partial B}\left(\sum_{m=1}^{k}(-1)^{m}C_{n-m}^{k-m} \frac{(n-k)(k+1)}{2(m+1)(n-m)} \right)|\nabla u|^{2} \sigma_{m}\left(D^{2} u\right)d A \\
& +O(\epsilon)\|\nabla u\|_{L^{2}}^{2}+O(\epsilon)\|u\|_{L^{2}}^{2}. 
\end{aligned}
\end{equation}
Without loss of generality, we may assume that the body $K$ is symmetric with respect to the line $\left\{t e_{1}: t \in \mathbb{R}\right\}$. Let $V:[0, \pi] \rightarrow(-1, \infty)$ be the function such that $V(\theta):=u(\cos (\theta), \sin (\theta), 0,0, \ldots, 0)$. Notice that, since we assume that $\partial K$ is smooth, $\dot{V}(0)=\dot{V}(\pi)=0$. Whenever in a formula we have both $\theta$ and $x$, it is always assumed implicitly that $\theta=\theta(x)=$ $ \arccos \left(x_{1}\right)$. The function $u$ is always implicitly evaluated at $x$, while the function $V$ is always implicitly evaluated at $\theta$.

\noindent As a consequence of the coarea formula, we have
\begin{equation}
\int_{\partial B} f(\theta) \, dA
= \mathscr{H}^{n-1}\!\left(\partial B\right) 
\int_{0}^{\pi} f(\theta)\sin^{\,n-1}(\theta)\, d\theta
\end{equation}

\noindent for any continuous function $f:[0, \pi] \rightarrow \mathbb{R}$.

\noindent Standard computations yield the following formulae

\begin{equation}
|\nabla u|=|\dot{V}|, \quad \nabla^{2} u[\nabla u, \nabla u]=\dot{V}^{2} \ddot{V}, \quad \Delta u=\ddot{V}+(n-2) \dot{V} \frac{\cos (\theta)}{\sin (\theta)} .
\end{equation}
Recall that in $\mathbb{R}^{n+1}$, the spherical coordinate transformation is given by
\begin{equation}
\begin{aligned}
x_1 &= r \cos \theta_1, \\
x_2 &= r \sin \theta_1 \cos \theta_2, \\
%x_3 &= r \sin \theta_1 \sin \theta_2 \cos\theta_3, \\
&\quad \vdots \\
x_{n+1} &= r \sin \theta_1 \sin \theta_2 \cdots \sin \theta_{n-1}\sin \theta_n,
\end{aligned}
\end{equation}
where
$ r\ge 0, \theta_1, \ldots, \theta_{n-1}\in [0,\pi], \theta_n \in [0,2\pi)$.
The transformation of differentials from Cartesian coordinates to spherical coordinates in $\mathbb{R}^{n+1}$ can be expressed via the Jacobian matrix below. That is, 

\begin{equation}
\begin{bmatrix}
\displaystyle
\mathrm{d}r \\[6pt]
\mathrm{d}\theta_1 \\[6pt]
\vdots \\[6pt]
\mathrm{d}\theta_n
\end{bmatrix}
=
\begin{bmatrix}
\displaystyle \frac{\partial r}{\partial x_1} & \cdots &\displaystyle \frac{\partial r}{\partial x_{n+1}} \\[10pt]
\displaystyle \frac{\partial \theta_1}{\partial x_1} & \cdots &\displaystyle \frac{\partial \theta_1}{\partial x_{n+1}} \\[10pt]
\vdots & \ddots & \vdots \\[10pt]
\displaystyle \frac{\partial \theta_n}{\partial x_1} & \cdots & \displaystyle\frac{\partial \theta_n}{\partial x_{n+1}}
\end{bmatrix}
\begin{bmatrix}
\mathrm{d}x_1 \\[6pt]
\mathrm{d}x_2 \\[6pt]
\vdots \\[6pt]
\mathrm{d}x_{n+1}
\end{bmatrix}
=
\begin{bmatrix}
\displaystyle e_1 \\[10pt]
\displaystyle re_2\\[10pt]
\displaystyle r\sin \theta_1e_3\\[10pt]
\vdots \\[10pt]
\displaystyle r\sin \theta_1 \sin \theta_2 \sin \theta_3 \cdots \sin \theta_{n-1} e_n
\end{bmatrix}
\begin{bmatrix}
\mathrm{d}x_1 \\[6pt]
\mathrm{d}x_2 \\[6pt]
\vdots \\[6pt]
\mathrm{d}x_{n+1}
\end{bmatrix}
\end{equation}
\noindent where $e_1^T, e_2^T, \ldots, e_{n+1}^T$ is an orthonormal basis of $\mathbb{R}^{n+1}$.

\noindent As a result, we have 
\begin{equation}
\left(
\begin{array}{ccc}
e_1^{T}, & e_2^{T}, & \cdots \quad e_{n+1}^{T}
\end{array}
\right)
\begin{bmatrix}
1 & 0 & 0 & \cdots & 0 \\
0 & \dfrac{1}{r} & 0 & \cdots & 0 \\
0 & 0 & \dfrac{1}{r \sin \theta_1} & \cdots & 0 \\
\vdots & \vdots & \vdots & \ddots & \vdots \\
0 & 0 & 0 & \cdots & \dfrac{1}{r \sin \theta_1 \sin \theta_2 \cdots \sin \theta_{n-1}}
\end{bmatrix}
\begin{bmatrix}
\dfrac{\partial}{\partial r} \\[8pt]
\dfrac{\partial}{\partial \theta_1} \\[8pt]
\vdots \\[8pt]
\dfrac{\partial}{\partial \theta_n}
\end{bmatrix}
=
\begin{bmatrix}
\dfrac{\partial}{\partial x_1} \\[8pt]
\dfrac{\partial}{\partial x_2} \\[8pt]
\vdots \\[8pt]
\dfrac{\partial}{\partial x_{n+1}}
\end{bmatrix}.
\end{equation}
Thus, we can use 
\begin{equation}
\begin{bmatrix}
\dfrac{\partial}{\partial r} \\[8pt]
\displaystyle\frac{1}{r}\dfrac{\partial}{\partial \theta_1} \\[8pt]
\vdots \\[8pt]
\displaystyle\frac{1}{r\sin \theta_1 \sin \theta_2 \cdots \sin \theta_{n-1}}\dfrac{\partial}{\partial \theta_n}
\end{bmatrix}
\end{equation}
as an orthonormal basis. Since $u$ is just a function of $\theta_1=\theta$, standard computations yield the Hessian of $u$ on the sphere 
\begin{equation}
D^2u=\begin{bmatrix}
0 & -u_{\theta} & & & \\
 -u_{\theta} & u_{\theta\theta} & & & \\
 & & \dfrac{\cos\theta}{\sin\theta}\, u_{\theta} & & \\
 & & & \ddots & \\
 & & & & \dfrac{\cos\theta}{\sin\theta}\, u_{\theta}
\end{bmatrix}.
\end{equation}
Since the computation of $\sigma_k(D^2u)$ is independent of the choice of the orthonormal coordinates, we deduce that for $k\geq 2$,
\begin{equation}\label{eqn:8.12}
\sigma_k\bigl(D^2 u\bigr)
= C_{n-1}^{\,k-1}\, \frac{\cos^{\,k-1}\theta}{\sin^{\,k-1}\theta}\, u_{\theta\theta}\, (u_{\theta})^{\,k-1}
+ C_{n-1}^{\,k-2}\, \frac{\cos^{\,k-2}\theta}{\sin^{\,k-2}\theta}\, (u_{\theta})^{\,k}
+ C_{n-1}^{\,k}\, \frac{\cos^{\,k}\theta}{\sin^{\,k}\theta}\, (u_{\theta})^{\,k}.
\end{equation}

\noindent Using these computations, the term containing $\sigma_k(h)$ appearing in (\ref{k}) becomes
\begin{equation}\label{F}
\begin{aligned}
& (-1)^k \int_{\partial B} \frac{1}{2}\, |\nabla u|^2\, \sigma_k\bigl(D^2 u\bigr)\, dA \\
=& (-1)^k\, \Bigl(\frac{1}{2k}\Bigr)\, \int_{\partial B} |\nabla u|^2\, u^i_j\, \bigl[T_{k-1}\bigr]_i^j\bigl(D^2 u\bigr)\, dA \\
=& (-1)^k\, \frac{-2}{2k}\, \int_{\partial B} u^i\, u_s\, u^s_j\, \bigl[T_{k-1}\bigr]_i^j\bigl(D^2 u\bigr)\, dA \;+\; \, O(\varepsilon)\,\|\nabla u\|_{L^2}^2 \\
=& (-1)^{k+1}\, \frac{1}{k}\, \int_{\partial B} u_2\, u_2\, u^2_j\, \bigl[T_{k-1}\bigr]_2^j\bigl(D^2 u\bigr)\, dA \;+\; \, O(\varepsilon)\,\|\nabla u\|_{L^2}^2 \\
=& (-1)^{k+1}\, \frac{1}{k}\, \int_{\partial B} u^2_{\theta}\, u_{\theta\theta}\, \bigl[T_{k-1}\bigr]_2^2\bigl(D^2u\bigr) \, dA \;+\; \, O(\varepsilon)\,\|\nabla u\|_{L^2}^2 \\
=& (-1)^{k+1}\, \frac{1}{k}\, \int_{\partial B} u_{\theta}^2\, u_{\theta\theta}\, \frac{1}{(k-1)!}\, C_{n-1}^{k-1}\, \bigl(\frac{\cos\theta}{\sin\theta}\bigr)^{k-1}\,(k-1)! u_{\theta}^{k-1} \, dA \;+\; \, O(\varepsilon)\,\|\nabla u\|_{L^2}^2 \\
=& (-1)^{k+1}\, \frac{1}{k}\, C_{n-1}^{k-1}\, \int_{\partial B} u_{\theta}^{k+1}\, u_{\theta\theta}\,\bigl(\frac{\cos\theta}{\sin\theta}\bigr)^{k-1}\, dA \;+\;\, O(\varepsilon)\,\|\nabla u\|_{L^2}^2.
\end{aligned}
\end{equation}

\noindent Using the coarea formula, we know that the integral is equal to 
\begin{equation}\label{G}
 (-1)^{k+1}\, \frac{1}{k}\, C_{n-1}^{k-1}\, \int_0^{\pi} u_{\theta}^{k+1}\, u_{\theta\theta}\,\cos^{k-1} \theta\sin ^{n-k}\theta\, d\theta \;+\; \, O(\varepsilon)\,\|\nabla u\|_{L^2}^2.
\end{equation}
Let $w(\theta)=\cos^{k-1}\theta\,\sin^{\,n-k}\theta$.  
Since $u_\theta^{k+1}u_{\theta\theta}=\dfrac{1}{k+2}(u_\theta^{k+2})_\theta$, we have
$$
\int_0^\pi u_\theta^{k+1}u_{\theta\theta}w(\theta)\,d\theta
= \frac{1}{k+2}\int_0^\pi (u_\theta^{k+2})_\theta\,w(\theta)\,d\theta
= -\frac{1}{k+2}\int_0^\pi u_\theta^{k+2}\,w_\theta(\theta)\,d\theta,
$$
where the boundary terms vanish.

\noindent A direct computation shows that
$$
w_\theta(\theta)
= (n-k)\cos^{k}\theta\,\sin^{\,n-k-1}\theta
-(k-1)\cos^{k-2}\theta\,\sin^{\,n-k+1}\theta.
$$

%\noindent Therefore, $$\begin{aligned} &(-1)^{k+1}\frac{1}{k}C_{n-1}^{k-1} \int_0^\pi u_\theta^{k+1}u_{\theta\theta}\cos^{k-1}\theta\sin^{\,n-k}\theta\,d\theta \\&\qquad =\frac{(-1)^k}{k(k+2)}C_{n-1}^{k-1}\int_0^\pi u_\theta^{k+2}\Bigl[(n-k)\cos^{k}\theta\,\sin^{\,n-k-1}\theta -(k-1)\cos^{k-2}\theta\,\sin^{\,n-k+1}\theta\Bigr]\,d\theta .\end{aligned}$$

\noindent Therefore,
\begin{equation}
\begin{aligned}\label{gggggggg}
&(-1)^{k+1}\frac{1}{k}C_{n-1}^{k-1}
\int_0^\pi 
u_\theta^{k+1}u_{\theta\theta}\cos^{k-1}\theta\sin^{\,n-k}\theta\,d\theta
+ O(\varepsilon)\|\nabla u\|_{L^2}^2 \\
\qquad =&
\frac{(-1)^k}{k(k+2)}C_{n-1}^{k-1}
\int_0^\pi 
u_\theta^{k+2}\Bigl[(n-k)\cos^{k}\theta\,\sin^{\,n-k-1}\theta
-(k-1)\cos^{k-2}\theta\,\sin^{\,n-k+1}\theta\Bigr]\,d\theta
+ O(\varepsilon)\|\nabla u\|_{L^2}^2 .
\end{aligned}
\end{equation}

\noindent Let $\theta_{0}=\theta_{0}(\varepsilon):=\varepsilon^{1-\frac{1}{k}}$, so that $\theta_{0}=\omega(\varepsilon)$ and $|u_{\theta}|=|\dot{V}(\theta)|=\omega(\varepsilon) \sin \theta$ for all $\theta_{0}<\theta<\pi-\theta_{0}$. Here $\omega(\varepsilon)$ denotes a quantity which tends to $0$ as $\varepsilon\rightarrow 0$. So we have
\begin{equation}
\begin{aligned}\label{ffffffffff}
\int_{0}^{\pi}\,\, u_{\theta}^{k+2}\,\,
\cos^k\theta
\sin^{n-k-1}\theta\,\mathrm{d}\theta 
&= \int_{0}^{\theta_{0}} u_{\theta}^{k+2}\,\sin^{n-k-1}\theta\,\mathrm{d}\theta
\;+\; \int_{\pi-\theta_{0}}^{\pi}\,u_{\theta}^{k+2}\, \,\sin^{n-k-1}\theta\,\mathrm{d}\theta \\
&\quad +\;\omega(\varepsilon)\,\int_{{\partial B}}\,|\nabla u|^{2}\,\mathrm{d}A.
\end{aligned}
\end{equation}
As a result, we only need to consider the interval in which $\theta$ is small. Using the identity $\cos^{2}\theta = 1 - \sin^{2}\theta$ and combining it with (\ref{ffffffffff}), we obtain that, for $k$ odd, the term in (\ref{gggggggg}) is equal to

\begin{equation}\label{dfg}
\begin{aligned}
&( (-1)^k\, \frac{n-k}{k}\, C_{n-1}^{k-1}\,+\omega(\epsilon)) \big[\int_0^{\theta_0}\frac{u_{\theta}^{k+2}}{k+2}\, \cos^k \theta\sin ^{n-k-1}\theta\, d\theta+\int_{\pi-\theta_0}^{\pi}\frac{u_{\theta}^{k+2}}{k+2}\, \cos^k \theta\sin ^{n-k-1}\theta\, d\theta\big]\;\\
&+\; \, \omega(\varepsilon)\,\|\nabla u\|_{L^2}^2\\
=&( (-1)^k\, \frac{n-k}{k}\, C_{n-1}^{k-1}\,+\omega(\epsilon)) \big[\int_0^{\theta_0}\frac{u_{\theta}^{k+2}}{k+2}\, \cos \theta\sin ^{n-k-1}\theta\, d\theta+\int_{\pi-\theta_0}^{\pi}\frac{u_{\theta}^{k+2}}{k+2}\, \cos \theta\sin ^{n-k-1}\theta\, d\theta\big]\;\\
&+\; \, \omega(\varepsilon)\,\|\nabla u\|_{L^2}^2\\
=&( (-1)^k\, \frac{n-k}{k}\, C_{n-1}^{k-1}\,+\omega(\epsilon)) \big[\int_0^{\theta_0}\frac{u_{\theta}^{k+2}}{k+2}\,\sin ^{n-k-1}\theta\, d\theta-\int_{\pi-\theta_0}^{\pi}\frac{u_{\theta}^{k+2}}{k+2}\,\sin ^{n-k-1}\theta\, d\theta\big]\;+\; \, \omega(\varepsilon)\,\|\nabla u\|_{L^2}^2
\end{aligned}
\end{equation}
\noindent and similarly for $k$ even, the term (\ref{gggggggg}) is 
\begin{equation}\label{dfgh}
 ((-1)^k\, \frac{n-k}{k}\, C_{n-1}^{k-1}\,+\omega(\epsilon))\big[\int_0^{\theta_0}\frac{u_{\theta}^{k+2}}{k+2}\, \sin ^{n-k-1}\theta\, d\theta+\int_{\pi-\theta_0}^{\pi}\frac{u_{\theta}^{k+2}}{k+2}\, \sin ^{n-k-1}\theta\, d\theta\big]+\; \, \omega(\varepsilon)\,\|\nabla u\|_{L^2}^2.
\end{equation}
Here we have used the fact that when $\sin \theta$'s exponent is bigger than $n-k-1$, the term is absorbed into $\omega(\varepsilon)\|u\|^2_{L^2}$.
We remark that $\displaystyle  \int_{0}^\pi u_\theta^{k+2} \sin^{n-k-1}\theta d\theta$ may not be comparable to $\displaystyle \int_{\partial B}|\nabla u|^{2}dA$,  because the asymptotic profile of $u_\theta$ and $\sin \theta$ may be different near $0$ and $\pi$. Before proceeding further, using (\ref{F}), (\ref{G}), (\ref{gggggggg}), (\ref{dfg}) and (\ref{dfgh}), we rewrite (\ref{k2}) as 
\begin{equation}
\begin{aligned}\label{k33}
I_{k}(\Omega)-I_{k}(B) 
&= C_{n}^{k}\frac{(n-k)k}{2n} 
   \int_{\partial B}\big(|\nabla u|^{2}-n u^{2}\big)\, dA \\[0.4em]
&\quad + \sum_{m=1}^{k} 
   (-1)^m\, \frac{n-m}{m}\, C_{n-1}^{m-1} C_{\,n-m}^{\,k-m} 
   \frac{(n-k)(k+1)}{(m+1)(n-m)} \\[-0.3em]
&\qquad \cdot \left(
      \int_{0}^{\theta_0} 
      \frac{u_{\theta}^{m+2}}{m+2}\, 
      \, \sin^{\,n-m-1}\theta\, d\theta
      + (-1)^m 
      \int_{\pi-\theta_0}^{\pi}
      \frac{u_{\theta}^{m+2}}{m+2}\,
      \, \sin^{\,n-m-1}\theta\, d\theta
   \right) \\[0.6em]
&\quad + O(\epsilon)\|\nabla u\|_{L^{2}}^{2}
        + O(\epsilon)\|u\|_{L^{2}}^{2} \\[0.8em]
&= C_{n}^{k}\frac{(n-k)k}{2n} 
   \int_{\partial B}\big(|\nabla u|^{2}-n u^{2}\big)\, dA \\[0.4em]
&\quad + 
   \left(
      \omega(\epsilon) 
      + \sum_{m=1}^{k}
        (-1)^m\, C_{n-1}^{k} C_{k}^{m}\,
        \frac{k+1}{(m+1)(m+2)}
   \right)\\[-0.3em]
&\qquad \cdot \left(
      \int_{0}^{\theta_0} 
      u_{\theta}^{m+2}\, 
      \, \sin^{\,n-m-1}\theta\, d\theta
      + (-1)^m 
      \int_{\pi-\theta_0}^{\pi}
      u_{\theta}^{m+2}\,
      \, \sin^{\,n-m-1}\theta\, d\theta
   \right) \\[0.6em]
&\quad + \omega(\epsilon)\|\nabla u\|_{L^{2}}^{2}
        + O(\epsilon)\|u\|_{L^{2}}^{2}.
%&\ge C_{n}^{k}\frac{(n-k)k}{2n}   \int_{\partial B}\big(|\nabla u|^{2}-n u^{2}\big)\, dA \\[0.4em]&\quad + \left(\omega(\epsilon) + \sum_{m=1}^{k}(-1)^m\, C_{n-1}^{k} C_{k}^{m}\, \frac{k+1}{(m+1)(m+2)}\right)\\[-0.3em]&\qquad \cdot \left(\int_{0}^{\theta_0} u_{\theta}^{m+2}\, \, \sin^{\,n-m-1}\theta\, d\theta- \int_{\pi-\theta_0}^{\pi}u_{\theta}^{m+2}\,\, \sin^{\,n-m-1}\theta\, d\theta\right) \\[0.6em]&\quad + \omega(\epsilon)\|\nabla u\|_{L^{2}}^{2}+ O(\epsilon)\|u\|_{L^{2}}^{2}.
\end{aligned}
\end{equation}
In order to proceed, we estimate $u_\theta^m $.
\noindent Recall that 
\begin{align}\label{1.5.1}
    \sigma_{k}(h) &= 
    \frac{\sum_{m=0}^{k} (-1)^{m} C_{n-m}^{k-m} \left((1+u)^{2} \sigma_{m}\left(D^{2} u\right) + \frac{n+k-2m}{n-m} u^{i} u_{j}\left[T_{m}\right]_{i}^{j}\left(D^{2} u\right)\right)}
    {\left((1+u)^{2}+|\nabla u|^{2}\right)^{\frac{k+2}{2}}}.
\end{align}
\begin{Lem} 
For $0\leq \theta \leq \theta_0$, we have
\begin{equation} \label{ineq}
\sum_{j=1}^{m}(-1)^{j}(1+\omega(\epsilon,j)) u_\theta^j C_m^{m-j}\sin^{m-j}\theta  
\ge (1+\omega(\epsilon))\Bigl\{-\min\left(\frac{\int_{0}^{\theta} 
\frac{m}{C_{n-1}^{m-1}}\sigma_m(h)^-\sin^{n-1}\tau\,d\tau}{\theta^{\,n-m}},B(n,m)\epsilon^{m}\right)
-\theta^m\Bigr\},
\end{equation}
\noindent where $B(n,m)$ is a constant depending on $n,m$.
\end{Lem}

\noindent\begin{p}
Using (\ref{1.5.1}), we have 
\[
(-1)^m \sigma_m(h)=  (1+\omega(\epsilon,m))\Bigl(\sum_{j=0}^m(-1)^{m+j} 
C_{n-j}^{m-j} \sigma_j(D^2u)\Bigr).
\]
Next, using (\ref{eqn:8.12}), we have
\[
\begin{aligned}
(-1)^m \sigma_m(h)\,\sin^{n-1}\theta
&= (1+\omega(\epsilon,m)) \sum_{j=0}^m (-1)^{m+j} C_{n-j}^{m-j} 
\sigma_j(D^2u)\, \sin^{n-1} \theta \\
&= (1+\omega(\epsilon,m)) \sum_{j=0}^m (-1)^{m+j} C_{n-j}^{m-j} \Big(
    C_{n-1}^{j-1}\cos^{j-1}\theta \,\sin^{n-j}\theta\, u_{\theta\theta}\,u_{\theta}^{\,j-1} \\
&\qquad\qquad\qquad
    + C_{n-1}^{j-2}\cos^{j-2}\theta \,\sin^{n-j+1}\theta\, u_{\theta}^{\,j}
    + C_{n-1}^j \cos^{j}\theta \,\sin^{n-1-j}\theta \,u_{\theta}^j
\Big).
\end{aligned}
\]

\noindent On the other hand, for $1\leq m \leq k$,
\begin{equation}
\begin{aligned}\label{hhhhhh}
&\frac{d}{d\theta}\Bigl(\sum_{j=1}^m(-1)^{m-j}\frac{C_{n-1}^{m-1}}{m} 
u_\theta^j C_m^{m-j}\sin^{n-j}\theta \ \cos^{j-1}\theta\Bigr)\\
&=\sum_{j=1}^m(-1)^{m-j}\frac{C_{n-1}^{m-1}}{m}\, C_m^{\,m-j}\,
\Bigl(
j\,u_\theta^{\,j-1}u_{\theta\theta}\,\sin^{\,n-j}\theta \,\cos^{\,j-1}\theta \\
&\qquad\qquad
+(n-j)u_\theta^{\,j}\,\sin^{\,n-j-1}\theta \,\cos^{\,j}\theta
-(j-1)u_\theta^{\,j}\,\sin^{\,n-j+1}\theta \,\cos^{\,j-2}\theta
\Bigr).
\end{aligned}
\end{equation}

\noindent Using the binomial identities,
\[
\frac{C_{n-1}^{m-1}}{m}\,C_m^{\,m-j}\,j
= C_{n-j}^{m-j} C_{n-1}^{j-1},\qquad
\frac{C_{n-1}^{m-1}}{m}\,C_m^{\,m-j}\,(n-j)
= C_{n-j}^{m-j} C_{n-1}^{j},
\]
we compare (\ref{hhhhhh}) with the expansion of $(-1)^m\sigma_m(h)\sin^{n-1}\theta$ above,
and noting that the $j=0$ term equals $(-1)^m C_n^m\sin^{n-1}\theta$, we obtain for $1\leq m\leq k$ (note that $\theta=\omega(\epsilon)$),
\begin{equation}
\begin{aligned}
&\frac{d}{d\theta}\Bigl(\sum_{j=1}^m(-1)^{m-j}\frac{C_{n-1}^{m-1}}{m} \,C_m^{\,m-j}\,
u_\theta^j \sin^{\,n-j}\theta \ \cos^{\,j-1}\theta\Bigr)\\
=&(1+\omega(\epsilon,m))\Bigl\{
((-1)^m\sigma_m(h)-(-1)^m C_n^m)\sin^{n-1}\theta
+\omega(\varepsilon)\cdot\Bigl(\sum_{j=0}^m u_\theta^j \sin^{\,n-j}\theta\Bigr)\Bigr\}.
\end{aligned}
\end{equation}

\noindent Hence,
\begin{equation}
\begin{aligned}\label{jklop}
&\frac{d}{d \theta}\Bigl(\sum_{j=1}^{m}(-1)^{j}u_\theta^j 
C_m^{m-j}\sin^{n-j}\theta \ \cos^{j-1}\theta \Bigr) \\
\ge&(1+\omega(\epsilon,m))\Bigl\{ 
\frac{m}{C_{n-1}^{m-1}}(-\sigma_m(h)^- -C_n^m)\sin^{n-1} \theta
+\omega(\varepsilon)\cdot 
\Bigl(\sum_{j=0}^m u_\theta^j \sin^{n-j}\theta\Bigr)\Bigr\}.
\end{aligned}
\end{equation}

\noindent Integrating this over $[0,\theta]$,
\begin{equation}
\sum_{j=1}^{m}(-1)^{j}(1+\omega(\epsilon,j))u_\theta^j 
C_m^{m-j}\sin^{n-j}\theta \ \cos^{j-1}\theta 
\ge(1+\omega(\epsilon))\Bigl\{ 
-\int_{0}^{\theta} \frac{m}{C_{n-1}^{m-1}}\sigma_m(h)^-\sin^{n-1}\tau\,d\tau 
-\theta^n\Bigr\}.
\end{equation}

\noindent Simplifying, we get
\begin{equation}
\sum_{j=1}^{m}(-1)^{j}(1+\omega(\epsilon,j))u_\theta^j 
C_m^{m-j}\sin^{m-j}\theta  
\ge(1+\omega(\epsilon))\Bigl\{
-\frac{\int_{0}^{\theta} 
\frac{m}{C_{n-1}^{m-1}}\sigma_m(h)^-\sin^{n-1}\tau\,d\tau}{\theta^{\,n-m}} 
-\theta^m\Bigr\}.
\end{equation}
Furthermore,  note that for $1 \leq j \leq m$, $u_{\theta}^j \sin^{m-j} \theta \leq \epsilon^{m}$ since $0 \leq \theta \leq \theta_0=\epsilon$ and $||u||_{C^1} \leq \epsilon$, so $\sum_{j=1}^{m}(-1)^{j}(1+\omega(\epsilon,j)) u_\theta^j C_m^{m-j}\sin^{m-j}\theta \ge -B(n,m)\epsilon^{m}$ for some positive constant $B(n,m)$. Thus the lemma is proved.
\end{p}

\begin{rem}
Note that if we integrate (\ref{jklop}) over $[\pi-\theta,\pi]$ and argue similarly, we can get that 
for $\pi-\theta_0\leq \theta \leq \pi$, we have
\begin{equation} \label{ineq2}
\begin{aligned}
&-\sum_{j=1}^{m}(-1)^{j}(1+\omega(\epsilon,j)) u_\theta^j(\theta) C_m^{m-j}\sin^{m-j}(\theta)  \\
&\ge (1+\omega(\epsilon))\Bigl\{-\min\left(\frac{\int_{\pi-\theta}^{\pi} 
\frac{m}{C_{n-1}^{m-1}}\sigma_m(h)^-\sin^{n-1}\tau\,d\tau}{(\pi-\theta)^{\,n-m}},B(n,m)\epsilon^{m}\right)
-(\pi-\theta)^m\Bigr\},
\end{aligned}
\end{equation}
This is exactly what we need to control the term in \eqref{k33} that integrates from $\pi-\theta_{0}$ to $\pi$.

\end{rem}
We now substitute \eqref{ineq} and \eqref{ineq2} into \eqref{k33} inductively. After the first substitution, we can cancel all terms containing $u_\theta^{\,k+2}\sin^{\,n-1-k}\theta$, leaving only terms with a lower power of $u_\theta$. We then continue to use (\ref{ineq}) to eliminate the terms containing $u_\theta^{\,k+1}\sin^{\,n-k}\theta$, and proceed in this way until the only term left is $u_\theta^{\,2}\sin^{\,n-1}\theta$. We note that during this elimination process, when $u_\theta^{\,m+2}\sin^{\,n-1-m}\theta$ becomes the term with the highest power of $u_\theta$, for $m<k$, its coefficient is
$$
\frac{C_{n-1}^k}{k+2}(C_{k+2}^{\,m+2\,}\;-\;\sum_{t=0}^{\,k-m-1}(t+1)C_{k-t}^{\,m}\,).
$$

\noindent Let us denote 
$$
T := C_{k+2}^{\,m+2\,}\;-\;\sum_{t=0}^{\,k-m-1}(t+1)C_{k-t}^{\,m}.
$$
To evaluate $T$, 
 set $j = k-t$, and write
\begin{equation*}
\begin{split}
\sum_{t=0}^{k-m-1}(t+1)C_{k-t}^{m}
=& \sum_{j=m+1}^k (k-j+1)C_{j}^{m}\\
=& (k+1)\sum_{j=m+1}^k C_{j}^{m} - \sum_{j=m+1}^k jC_{j}^{m}.
\end{split}
\end{equation*}

\noindent Then we simplify each term. By the identity,
$$
\sum_{j=m+1}^k C_{j}^{m} = C_{k+1}^{m+1} - 1.
$$

\noindent Also, using $jC_{j}^{m} = mC_{j}^{m}+(m+1)C_{j}^{m+1}$,
$$
\sum_{j=m+1}^k jC_{j}^{m}
= m\left(C_{k+1}^{m+1}-1\right) + (m+1)C_{k+1}^{m+2}.
$$

\noindent Finally, we put everything together. Thus
\begin{equation}
\begin{aligned}
\sum_{t=0}^{k-m-1}(t+1)C_{k-t}^{m}
= &(k+1)\big(C_{k+1}^{m+1}-1\big) - m\big(C_{k+1}^{m+1}-1\big) - (m+1)C_{k+1}^{m+2}\\
=&
 (k+1-m)C_{k+1}^{m+1} - (m+1)C_{k+1}^{m+2}-(k-m+1).
 \end{aligned}
\end{equation}

\noindent Finally, we evaluate $T$. Since
$$
T = C_{k+2}^{m+2} - \left[ (k+1-m)C_{k+1}^{m+1} - (m+1)C_{k+1}^{m+2} + (k-m+1) \right],
$$
using Pascal’s rule $C_{k+2}^{m+2} = C_{k+1}^{m+2}+C_{k+1}^{m+1}$, we obtain
$$
T = (m+2)C_{k+1}^{m+2} - (k-m)C_{k+1}^{m+1} - (k-m+1).
$$
By the identity
$$
(k-m)C_{k+1}^{m+1} = (m+2)C_{k+1}^{m+2},
$$
\noindent the binomial terms cancel neatly, and we are left with
$$
T =-( k - m + 1).
$$

\noindent So for integers $m<k$, the following identity holds:
$$
C_{k+2}^{m+2} - C_{k}^{m} - 2C_{k-1}^{m} - \cdots - (k-m)C_{m+1}^{m} = -(k - m + 1).
$$
As a result, we get that 
\begin{equation}
\begin{aligned}\label{k3}
I_{k}(\Omega)-I_{k}(B)\ge&C_{n}^{k} \frac{(n-k) k}{2 n} \int_{\partial B}|\nabla u|^{2}-n u^{2}dA+\\
&\quad + 
   \left(
      \omega(\epsilon) 
      + \sum_{m=1}^{k}
        (-1)^m\, C_{n-1}^{k} C_{k}^{m}\,
        \frac{k+1}{(m+1)(m+2)}
   \right)\\[-0.3em]
&\qquad \cdot \left(
      \int_{0}^{\theta_0} 
      u_{\theta}^{m+2}\, 
      \, \sin^{\,n-m-1}\theta\, d\theta
      - 
      \int_{\pi-\theta_0}^{\pi}
      u_{\theta}^{m+2}\,
      \, \sin^{\,n-m-1}\theta\, d\theta
   \right) \\%[0.6em] \\
\ge &C_{n}^{k} \frac{(n-k) k}{2 n} \int_{\partial B}|\nabla u|^{2}-n u^{2}dA\\
&-\sum_{m=1}^{k}\frac{C_{n-1}^km}{k+2}(\int_0^{\theta_0} u_{\theta}^{2}\,\sin ^{n-1}\theta\, d\theta+\int_{\pi-\theta_0}^{\pi} u_{\theta}^{2}\,\sin ^{n-1}\theta\, d\theta) \\
&\;-\sum_{m=1}^{k} c(n,m)\,
\int_{0}^{\theta_0}
u_{\theta}^{2}\,
\min\left(\frac{\displaystyle \int_{0}^{\theta} \sigma_m(h)^-\sin^{n-1}\tau\, d\tau}{\theta^{\,n-m}}, B(n,m)\epsilon^{m}\right)\,\theta^{\,n-m-1}\, d\theta \\
&-\int_{\pi-\theta_0}^{\pi}
u_{\theta}^{2}\,
\min\left(\frac{\displaystyle \int_{\pi-\theta}^{\pi} \sigma_m(h)^-\sin^{n-1}\tau\, d\tau}{(\pi-\theta)^{\,n-m}}, B(n,m)\epsilon^{m}\right)\,(\pi-\theta)^{\,n-m-1}\, d\theta\\
& +O(\epsilon)\|\nabla u\|_{L^{2}}^{2}+O(\epsilon)\|u\|_{L^{2}}^{2} \\
\ge& \left(C_{n}^{k} \frac{(n-k)k}{2 n}-(\sum_{m=1}^{k}\frac{C_{n-1}^km}{k+2})\right) \int_{\partial B}|\nabla u|^{2}-n u^{2}dA\\
-\sum_{m=1}^{k} c(n,m)&\,
\int_{0}^{\pi}
(u_{\theta}(\tau)^{2}+u_{\theta}(\pi-\tau)^2)\,
\min\left(\frac{\displaystyle \int_{0}^{\pi} \sigma_m(h)^-\sin^{n-1}\tau'\, d\tau'}{\tau} 
,B(n,m)\epsilon^{m} \tau^{n-m-1}\right)d\tau\\
& +O(\epsilon)\|\nabla u\|_{L^{2}}^{2}+O(\epsilon)\|u\|_{L^{2}}^{2} \\
\ge&\left(C_{n-1}^k\frac{k}{2(k+2)}\right) \int_{\partial B}|\nabla u|^{2}-n u^{2}dA\\
&\;-2\epsilon^2\sum_{m=1}^{k} c(n,m)\,
\int_{0}^{\pi}
\min\left(\frac{\displaystyle \int_{0}^{\pi} \sigma_m(h)^-\sin^{n-1}\tau\, d\tau}{\theta} 
,B(n,m)\epsilon^{m} \theta^{n-m-1}\right)d\theta\\
& +O(\epsilon)\|\nabla u\|_{L^{2}}^{2}+O(\epsilon)\|u\|_{L^{2}}^{2},
\end{aligned}
\end{equation}
for some $c(n,m)>0$.
Now, the only barrier is the term 
\begin{equation} \label{kp}
    \int_{0}^{\pi}
\min\left(\frac{\displaystyle \int_{0}^{\pi} \sigma_m(h)^-\sin^{n-1}\tau\, d\tau}{\theta} 
,B(n,m)\epsilon^{m} \theta^{n-m-1}\right)d\theta.
\end{equation}
We choose $\theta_1 = (\frac{A_m}{\epsilon^m})^{\frac{1}{n-m}}$, where $A_m=\int_{0}^{\pi} \sigma_m(h)^-\sin^{n-1}\tau\,d\tau=(1+O(\epsilon))\int_{\partial \Omega}\sigma_m(h)^-d\mu$, and then divide the original integral (\ref{kp}) into two parts. We get
\begin{equation} \label{kkkkkk}
\begin{aligned}
   &\quad \int_{0}^{\pi}
\min\left(\frac{\displaystyle \int_{0}^{\pi} \sigma_m(h)^- \sin^{n-1}\tau\, d\tau}{\theta} 
,B(n,m)\epsilon^{m} \theta^{n-m-1}\right)d\theta\\
&\leq \int_{\theta_1}^{\pi}
\frac{\displaystyle \int_{0}^{\pi} \sigma_m\!(h)^-\sin^{n-1}\tau\, d\tau}{\theta} 
 d\theta+\int_{0}^{\theta_1}
B(n,m)\epsilon^{m} \theta^{n-m-1}d\theta\\
&\leq  \left(\int_{\theta_1}^{\pi}
\frac{\displaystyle \int_{0}^{\pi} \sigma_m(h)^-\sin^{n-1}\tau\, d\tau}{\theta} 
 d\theta+\int_{0}^{\theta_1}
B(n,m)\epsilon^{m} \theta^{n-m-1}d\theta\right)\\
&\leq \left((ln(\pi)-ln(\theta_1))
\int_{0}^{\pi} \sigma_m(h)^-\sin^{n-1}\tau\, d\tau+
B(n,m)\epsilon^{m} \theta_1^{n-m}\right).\\
\end{aligned}
\end{equation}
\noindent Using $\theta_1 = (\frac{A_m}{\epsilon^m})^{\frac{1}{n-m}}$ and (\ref{kkkkkk}), it is easy to see that
\begin{equation} \label{kmj}
\int_{0}^{\pi}
\min\left(\frac{\displaystyle \int_{0}^{\theta} \sigma_m(h)^-\, d\tau}{\theta} 
,B(n,m)\epsilon^{m} \theta^{n-m-1}\right)d\theta\leq O(\ln \epsilon) \int_{0}^{\pi} \sigma_m(h)^-\sin^{n-1}\tau \,d\tau.
\end{equation}
\noindent Thus, combining (\ref{k3}) and (\ref{kmj}), we get 
\begin{equation}
\begin{aligned}
&I_{k}(\Omega)-I_{k}(B)\\
\geq &\left(C_{n-1}^k \frac{k}{2(k+2)}\right) \int_{\partial B}|\nabla u|^{2}-n u^{2}dA+O(\epsilon)\sum_{m=1}^{k} \int_{\partial \Omega} \sigma_m(h)^- d\mu+O(\epsilon)\|\nabla u\|_{L^{2}}^{2}+O(\epsilon)\|u\|_{L^{2}}^{2}. 
\end{aligned}
\end{equation}
\noindent As before, by Lemma \ref{lem:3.1}, we may replace $u$ with $u_2$ without affecting the proof argument, where $u_1$ and $u_2$ are the decomposition components of $u$ as previously defined. After this replacement, the coefficient for $|\nabla u_2|^2$ is positive. This positivity enables the application of the stronger Poincaré inequality. Thus the proof is complete.

 \section{Counterexample}
\noindent
\begin{p}of Theorem \ref{thm:9.1}.
\noindent Note that we have  
\begin{equation}\label{ch9}
 \begin{aligned}
	\int_{M} \sigma_{k}(h)\, d\mu 
	&=  \int_{\partial B} \Biggl\{ C_{n}^{k} + C_{n}^{k}(n-k) u 
	+ C_{n}^{k} \frac{(n-k)(n-k-1)}{2} u^{2} \\
	&\qquad + \sum_{m=0}^{k} (-1)^{m} C_{n-m}^{k-m} 
	\frac{(n-k)(k+1)}{2(m+1)(n-m)} |\nabla u|^{2} 
	\sigma_{m}\bigl(D^{2} u\bigr) \Biggr\}\, dA \\
	&\quad + o(\varepsilon)\|\nabla u\|_{L^{2}}^{2} 
	+ o(\varepsilon)\|u\|_{L^{2}}^{2}.
 \end{aligned}
\end{equation}

\noindent Consider now the integral  
$$
	\int_{\mathbb{R}^{n}} |\nabla u|^{2} \sigma_{m}\bigl(D^{2} u\bigr)\, dx.
$$
A function on the sphere is called \emph{radial} if it depends only on the distance from a fixed point $\bar x \in S^{n}$. In this case, we refer to $\bar x$ as the \emph{origin} of the radial function. Fix a large integer $\kappa \gg \varepsilon^{-3}$, and choose $q = c \kappa^{n}$ points $(x_i)_{i=1}^{q}$ on $S^{n}$ (with $c = c(n)$ universal) such that for any $i \neq j$,
$$
	|x_i - x_j| \ge 2\kappa^{-1}.
$$

\noindent Let $f:[0,\infty)\to\mathbb{R}$ be a smooth function supported in $[0,\kappa^{-1}]$ and satisfying
$$
	-\frac{\varepsilon}{2} \le f \le 0,\qquad 
	0 \le f' \le \frac{\varepsilon}{2},\qquad 
	f' = \frac{\varepsilon}{2}\ \text{on}\ 
	\Bigl[\frac{1}{2}\kappa^{-1}, \frac{3}{4}\kappa^{-1}\Bigr].
$$
For each $1 \le i \le q$, define a radial function $u_i:S^{n}\to\mathbb{R}$ with profile $f$ and origin at $x_i$. The functions $(u_i)_{i=1}^{q}$ have disjoint supports, each contained in a small geodesic ball of radius $\kappa^{-1}$. Moreover, $\|u_i\|_{C^1} \le \varepsilon$ for all $i$. In this situation, using the results of \cite{TSO89}, one has
\begin{equation}
 \begin{aligned}
	\sigma_{k}\bigl(D^2 u\bigr) 
	&= C_{n-1}^{k-1} f^{\prime \prime}\left(\frac{f^{\prime}}{r}\right)^{k-1} 
	+ C_{n-1}^{k}\left(\frac{f^{\prime}}{r}\right)^{k} \\
	&= C_{n-1}^{k-1} r^{-n+1} 
	\left(\frac{r^{n-k}}{k}\,(f^{\prime})^{k}\right)^{\prime}.
 \end{aligned}
\end{equation}
Consequently,
\begin{equation}
 \int_{\mathbb{R}^{n}} |Du|^{2} \sigma_{k}\bigl(D^{2} u\bigr) \, dx 
 = \int_{0}^{+\infty} |f'|^{2} 
 \, C_{n-1}^{k-1} r^{-n+1} 
 \left(\frac{r^{n-k}}{k}(f')^{k}\right)' r^{n-1} dr.
\end{equation}

\noindent Expanding this expression gives
\begin{equation}
 \frac{C_{n-1}^{k-1}}{k} 
 \int_{0}^{+\infty}(n-k) r^{n-k-1}(f')^{k+2}dr 
 + \frac{C_{n-1}^{k-1}}{k} 
 \int_{0}^{+\infty} r^{n-k}k (f')^{k+1}f''dr.
\end{equation}

\noindent Using integration by parts, this further simplifies to
\begin{equation}\label{radial}
 \begin{aligned}
	&C_{n-1}^{k-1}(n-k)\frac{2}{k(k+2)}
	\int_0^{+\infty}r^{n-k-1}(f')^{k+2}\,dr \\
	&\qquad = \alpha(n,k)\,\varepsilon^{k+2}\,\kappa^{k-n},
 \end{aligned}
\end{equation}
where $\alpha(n,k)>0$ is a constant depending on $n$, $k$, and the specific choice of $f$.

Let
$$
	u = \sum_{i=1}^{q} u_i.
$$
Let $K \subset \mathbb{R}^{n+1}$ be the star–shaped domain such that for each $x \in S^{n}$, the point $(1+u(x))x$ lies on $\partial K$. Then $K$ is a $C^{1}$ $\varepsilon$-perturbation of the unit ball.

\medskip

\noindent Combining \eqref{ch9} and \eqref{radial}, we obtain
\begin{equation}
 \int_{\partial \Omega} \sigma_{k}(h)\, d\mu 
 = c(n,k) 
 + \sum_{m=0}^{k} (-1)^{m} \alpha(n,m) 
 \varepsilon^{m+2}\,\kappa^{m}.
\end{equation}
Here $c(n,k)$ denotes a constant depending only on $n$ and $k$, and $\alpha(n,m)$ are positive constants depending on $n$ and $m$.  

Therefore, by choosing $\kappa$ sufficiently large and taking $k$ to be \emph{odd}, the alternating series on the right-hand side dominates the constant term, and we obtain a counterexample.
\end{p}
 
%\noindent Some text referencing the work~\cite{VW22} and \cite{GLAUDO2022108595}.
\bibliographystyle{plain} 
\bibliography{reference}  

\begin{thebibliography}{10}

\bibitem{AFM21}
Virginia Agostiniani, Mattia Fogagnolo, and Lorenzo Mazzieri.
\newblock Minkowski inequalities via nonlinear potential theory.
\newblock {\em Archive for Rational Mechanics and Analysis}, 244:1--35, 2022.

\bibitem{Ale37}
A.~Alexandroff.
\newblock Zur theorie der gemischten volumina von konvexen k{\"o}rpern. ii. neue ungleichungen zwischen den gemischten volumina und ihre anwendungen.
\newblock {\em Sbornik: Mathematics}, 44(6), 1937.

\bibitem{Ale38}
A.~Alexandroff.
\newblock Zur theorie der gemischten volumina von konvexen k{\"o}rpern. iii. die erweiterung zweier lehrs{\"a}tze minkowskis {\"u}ber die konvexen polyeder auf die beliebigen konvexen k{\"o}rper.
\newblock {\em Sbornik: Mathematics}, 45(1), 1938.

\bibitem{brendle2020isoperimetricinequalityminimalsubmanifold}
Simon Brendle.
\newblock The isoperimetric inequality for a minimal submanifold in {E}uclidean space.
\newblock {\em J. Amer. Math. Soc.}, 34(2):595--603, 2021.

\bibitem{brendle2023proofmichaelsimonsobolevinequalityusing}
Simon Brendle and Michael Eichmair.
\newblock Proof of the {M}ichael-{S}imon-{S}obolev inequality using optimal transport.
\newblock {\em J. Reine Angew. Math.}, 804:1--10, 2023.

\bibitem{Cabre1}
Xavier Cabr\'{e}.
\newblock Partial differential equations, geometry and stochastic control.
\newblock {\em Butl. Soc. Catalana Mat.}, 15(1):7--27, 2000.

\bibitem{Cabre2}
Xavier Cabr\'{e}.
\newblock Elliptic {PDE}'s in probability and geometry: symmetry and regularity of solutions.
\newblock {\em Discrete Contin. Dyn. Syst.}, 20(3):425--457, 2008.

\bibitem{CASTILLON201079}
Philippe Castillon.
\newblock Submanifolds, isoperimetric inequalities and optimal transportation.
\newblock {\em Journal of Functional Analysis}, 259(1):79--103, 2010.

\bibitem{ChangWangIMRN}
Sun-Yung~A. Chang and Yi~Wang.
\newblock Some higher order isoperimetric inequalities via the method of optimal transport.
\newblock {\em International Mathematics Research Notices}, 2014(24):6619--6644, 2013.

\bibitem{CHANG2013335}
Sun-Yung~Alice Chang and Yi~Wang.
\newblock Inequalities for quermassintegrals on k-convex domains.
\newblock {\em Advances in Mathematics}, 248:335--377, 2013.

\bibitem{Chavel}
Isaac Chavel.
\newblock {\em Isoperimetric inequalities}, volume 145 of {\em Cambridge Tracts in Mathematics}.
\newblock Cambridge University Press, Cambridge, 2001.
\newblock Differential geometric and analytic perspectives.

\bibitem{De_Rosa_2021}
Antonio De~Rosa and Stefano Gioffrè.
\newblock Absence of bubbling phenomena for non-convex anisotropic nearly umbilical and quasi-einstein hypersurfaces.
\newblock {\em Journal für die reine und angewandte Mathematik}, (780):1–40, 2021.

\bibitem{Evans1991635}
L.C. Evans and J.~Spruck.
\newblock Motion of level sets by mean curvature. i.
\newblock {\em Journal of Differential Geometry}, 33(3):635 – 681, 1991.

\bibitem{Fuglede}
Bent Fuglede.
\newblock Stability in the isoperimetric problem for convex or nearly spherical domains in {${\bf R}^n$}.
\newblock {\em Trans. Amer. Math. Soc.}, 314(2):619--638, 1989.

\bibitem{Gerhardt1990299}
Claus Gerhardt.
\newblock Flow of nonconvex hypersurfaces into spheres.
\newblock {\em Journal of Differential Geometry}, 32(1):299 – 314, 1990.

\bibitem{GLAUDO2022108595}
F.~Glaudo.
\newblock Minkowski inequality for nearly spherical domains.
\newblock {\em Adv. Math.}, 408(part B):Paper No. 108595, 33, 2022.

\bibitem{GUAN20091725}
Pengfei Guan and Junfang Li.
\newblock The quermassintegral inequalities for k-convex starshaped domains.
\newblock {\em Advances in Mathematics}, 221(5):1725--1732, 2009.

\bibitem{hormander1994notions}
Lars H\"{o}rmander.
\newblock {\em Notions of convexity}, volume 127 of {\em Progress in Mathematics}.
\newblock Birkh\"{a}user Boston, Inc., Boston, MA, 1994.

\bibitem{Huisken2001353}
Gerhard Huisken and Tom Ilmanen.
\newblock The inverse mean curvature flow and the riemannian penrose inequality.
\newblock {\em Journal of Differential Geometry}, 59(3):353 – 437, 2001.

\bibitem{Michael1973SobolevAM}
J.~H. Michael and L.~M. Simon.
\newblock Sobolev and mean‐value inequalities on generalized submanifolds of rn.
\newblock {\em Communications on Pure and Applied Mathematics}, 26:361--379, 1973.

\bibitem{Qiu}
Guohuan Qiu.
\newblock A family of higher-order isoperimetric inequalities.
\newblock {\em Communications in Contemporary Mathematics}, 17(03):1450015, 2015.

\bibitem{Scheuer+2025+381+405}
Julian Scheuer.
\newblock Stability from rigidity via umbilicity.
\newblock {\em Advances in Calculus of Variations}, 18(2):381--405, 2025.

\bibitem{TRUDINGER1994411}
Neil~S. Trudinger.
\newblock Isoperimetric inequalities for quermassintegrals.
\newblock {\em Ann. Inst. H. Poincar\'{e} C Anal. Non Lin\'{e}aire}, 11(4):411--425, 1994.

\bibitem{TSO89}
Kaising Tso.
\newblock On symmetrization and hessian equations.
\newblock {\em Journal d'Analyse Math{\'e}matique}, 52(1):94--106, 1989.

\bibitem{Urbas1990355}
John~I.E. Urbas.
\newblock On the expansion of starshaped hypersurfaces by symmetric functions of their principal curvatures.
\newblock {\em Mathematische Zeitschrift}, 205(1):355 – 372, 1990.

\bibitem{VW22}
Caroline VanBlargan and Yi~Wang.
\newblock Quantitative quermassintegral inequalities for nearly spherical sets.
\newblock {\em Communications in Contemporary Mathematics}, 26(06):2350026, 2024.

\bibitem{WangCaroline}
Caroline VanBlargan and Yi~Wang.
\newblock Stability of quermassintegral inequalities along inverse curvature flows.
\newblock {\em Calc. Var. Partial Differential Equations}, 2024.

\bibitem{wang2013michaelsimoninequalitieskthmean}
Yi~Wang.
\newblock Michael-{S}imon inequalities for {$k$}-th mean curvatures.
\newblock {\em Calc. Var. Partial Differential Equations}, 51(1-2):117--138, 2014.

\end{thebibliography}

\end{document}